\documentclass[11pt, english]{amsart}

\usepackage{amssymb}

\usepackage[utf8x]{inputenc}

\usepackage[T1]{fontenc}

\usepackage{libertine}

\usepackage{csquotes}

\usepackage{xcolor}

\colorlet{darkolive}{olive!60!black}
\usepackage[naturalnames=true, pageanchor=true, hyperfootnotes=false, colorlinks, allcolors=darkolive]{hyperref}

\usepackage[shortlabels]{enumitem}
\makeatletter
\renewcommand*{\descriptionlabel}[1]{%
  \let\orglabel\label
  \let\label\@gobble
  \phantomsection
  \edef\@currentlabel{#1}%
  \let\label\orglabel
  \hspace\labelsep \upshape\bfseries #1.
}
\makeatother

\usepackage{epigraph}
\usepackage[scr]{rsfso}

\theoremstyle{plain}
\newtheorem{proposition}{Proposition}[section]
\newtheorem{theorem}[proposition]{Theorem}
\newtheorem{corollary}[proposition]{Corollary}
\newtheorem{lemma}[proposition]{Lemma}

\theoremstyle{definition}
\newtheorem{definition}[proposition]{Definition}
\newtheorem{remark}[proposition]{Remark}

\newtheorem{example}[proposition]{Example}

\renewcommand{\phi}{\varphi}

\DeclareMathOperator{\Homeo}{Homeo} 
\DeclareMathOperator{\Aut}{Aut} 
\newcommand{\Autulb}{\mathrm{ULB}} 
\DeclareMathOperator{\Sym}{Sym} 
\DeclareMathOperator{\id}{id}
\newcommand{\Endo}{\mathcal{E}}
\newcommand{\Orth}{\mathrm{O}}
\newcommand{\conj}{\gamma}
\newcommand{\Reels}{\mathbf{R}}
\newcommand{\Rat}{\mathbf{Q}}
\newcommand{\Nat}{\mathbf{N}}
\newcommand{\Ent}{\mathbf{Z}}
\newcommand{\inv}{^{-1}}
\newcommand{\Diag}{\Delta}
\newcommand{\Unif}{\mathscr{U}}
\newcommand{\Subsetunif}{\mathscr{B}}
\newcommand{\Coar}{\mathscr{C}}
\newcommand{\Born}{\mathfrak{B}}
\newcommand{\Coll}{\mathfrak{S}}
\newcommand{\dist}{\mathrm{d}}
\newcommand{\Power}{\mathfrak{P}}
\newcommand{\Entsc}{\mathbf{E}}
\newcommand{\cvg}{\rightarrow}
\newcommand{\Lin}{\mathcal{L}}
\newcommand{\GL}{\mathrm{GL}}
\newcommand{\Left}{\mathscr{L}}
\newcommand{\Right}{\mathscr{R}}
\newcommand{\Roelcke}{\wedge}
\newcommand{\Upper}{\vee}

\newcommand{\Alex}[1]{\widehat{#1}}
\newcommand{\Alg}{\mathscr{A}}

\newcommand{\Measb}{\mathcal{M}_\mathrm{b}}
\newcommand{\Fam}{\mathscr{F}}
\newcommand{\Prob}{\mathscr{P}}

\newcommand{\ens}[2]{\left\{ #1 \,\,\, \middle| \,\,\, #2 \right\}} 
\newcommand{\symdiff}{\bigtriangleup}

\newcommand{\UnifLeft}{\Unif^\Left}
\newcommand{\UnifRight}{\Unif^\Right}
\newcommand{\UnifUpper}{\Unif^{\vee}}
\newcommand{\UnifRoelcke}{\Unif^{\wedge}}
\newcommand{\UnifConv}{\Unif^{\mathrm{conv}}}

\newcommand{\Uc}{\mathcal{E}}
\newcommand{\UcLeft}{\Uc^\Left}
\newcommand{\UcRight}{\Uc^\Right}
\newcommand{\UcUpper}{\Uc^{\vee}}

\newcommand{\Top}{\mathcal{T}}
\newcommand{\TopUpper}{\Top^{\vee}}
\newcommand{\TopLower}{\Top^{\wedge}}

\newcommand{\Graph}{\mathcal{G}}
\newcommand{\Gal}{\mathrm{Gal}}
\newcommand{\extens}[2]{#1{\!}:{\!}#2}

\newcommand{\Bornfin}{\Born_{\mathrm{f}}}
\newcommand{\BornKrull}{\Born_{\mathrm{Kr}}}
\newcommand{\Borndegree}{\Born_{\mathrm{d}}}

\newcommand{\TopKrull}{\Top_{\mathrm{Kr}}}
\newcommand{\Topdegree}{\Top_{\mathrm{d}}}

\newcommand{\Zentrum}{\mathcal{Z}}

\newcommand{\abs}[1]{\left| #1 \right|}

\DeclareMathOperator{\Isom}{Isom}
\DeclareMathOperator{\diam}{diam}

\newcommand{\bbkitgix}[3]{\cite[\textsc{ix}, \S~{#1}, n°~{#2}, #3]{Bourbaki_TG_V}}
\newcommand{\bbkitgx}[3]{\cite[\textsc{x}, \S~{#1}, n°~{#2}, #3]{Bourbaki_TG_V}}
\newcommand{\bbkitgii}[3]{\cite[\textsc{ii}, \S~{#1}, n°~{#2}, #3]{Bourbaki_TG_I}}

\newcommand{\bbkiav}[3]{\cite[\textsc{v}, \S~{#1}, n°~{#2}, #3]{Bourbaki_A_IV}}

\title[Uniformly locally bounded spaces]{Uniformly locally bounded spaces and the group of automorphisms of a topological group}
\author{Maxime Gheysens}

\address{Institüt für Geometrie, Technische Universität Dresden, 01062 Dresden, Germany}
\email{maxime.gheysens@tu-dresden.de}
\thanks{Work supported in part by the European Research Council Consolidator Grant no.~681207.}
\date{14 December 2019}

\begin{document}

\begin{abstract}
 We show that the topology of uniform convergence on bounded sets is compatible with the group law of the automorphism group of a large class of spaces that are endowed with both a uniform structure and a bornology, thus yielding numerous examples of topological groups. This encompasses many of previously known cases, such as homeomorphism groups of locally compact spaces, automorphism groups of first-order structures, general linear groups of normed spaces, automorphism groups of uniform spaces, etc. We also study the automorphism group of a topological group within this framework (generalizing previously known results on locally compact groups). As a particular outcome of this machinery, we show that the automorphism group of a Polish locally Roelcke-precompact group carries a natural Polish group topology.
\end{abstract}

\maketitle

Automorphism groups of mathematical structures form a standard source of examples of groups---indeed one which predates the abstract notion of groups. However, it is harder to use such a tool to provide examples of \emph{topological} groups, namely, to find a natural topology on these groups that makes the inversion map $g \mapsto g\inv$ and the composition map $(g, h) \mapsto gh$ continuous.

On the group $\Sym (X)$ of all bijections of a set $X$ , the most natural topology is the \emph{pointwise convergence topology}, namely the topology that the product $X^X$ (where $X$ is considered as a discrete space) induces on $\Sym(X)$. This topology is indeed compatible with the group structure of $\Sym(X)$ --- hence in particular with any subgroup $\Aut(X)$ if $X$ is endowed with some extra structure of interest. But it has two drawbacks with respect to yielding examples:
\begin{itemize}
	\item if $X$ was given with some topology, this piece of data is lost. On the other hand, considering the pointwise convergence topology with respect to that given topology is not often compatible with the group structure of $\Aut(X)$ (but see \cite{Gheysens_scatt} for an example where it does).
	\item all subgroups of $\Sym (X)$ share some common topological feature, for instance, being non-Archimedean (i.e., having a local base of open subgroups).
\end{itemize}

Actually, this pointwise convergence topology is closely related to automorphism groups of \emph{first-order} structures: they all appear as closed subgroups of $\Sym(X)$ (see \cite[4.1.4]{Hodges_1993}). The first purpose of this note is to give general examples of topological groups arising as automorphism groups of \emph{higher-order} structures, namely uniform spaces. 

Our target graal is, when $X$ itself is a topological group, to endow $\Aut(X)$ (the group of all its bicontinuous group automorphisms) with a compatible topology. Although it is not always possible to do so (excluding discrete and trivial topologies), the framework we propose will recover the previously known cases of discrete groups, locally compact groups and Banach spaces, and add to this list the locally Roelcke-precompact groups, a class of groups that attracted recent attention \cite{Zielinski_2019, Zielinski_2016}, see also \cite[\S~3.1]{Rosendal_coarse}. Outside the group realm, it also encompasses homeomorphism groups of locally compact spaces and isometry groups of metric spaces\footnote{Although, for the latter, not with their usual topology, see Section~\ref{sec:metricspaces}.}---as well as, as extreme cases, the familiar permutation group $\Sym(X)$ of a set (with the pointwise convergence topology) and the automorphism group of a uniform space (with the uniform convergence topology).

The common feature of these examples (invisible in the extreme cases) is that they carry two intertwined pieces of data: a uniform structure and a notion of boundedness such that some neighborhoods are bounded. We will call such spaces \emph{uniformly locally bounded spaces} (see Section~\ref{sec:ulbspace} for the formal definition). A natural topology to consider on maps between two such spaces is the topology of uniform convergence on bounded sets. The point of uniformly local boundedness is that this topology makes the composition map $G \times G \rightarrow G$ jointly continuous, when $G$ is the group of bijections that preserve the bounded sets and the uniform structure induced on the latter. Taking into account the inversion map, this is thus a source of examples of topological groups (Section~\ref{sec:topcomp} and in particular Theorem~\ref{th:comp}). Interestingly, this construction behaves well with respect to completeness (Section~\ref{sec:completeness}), thus paving the way to yielding Polish groups.

The variety of situations in which this general machinery can be applied is illustrated in Section~\ref{sec:examples}. But before that, we devote Section~\ref{sec:groups} to studying the particular case of the automorphism group $\Aut(X)$ of a topological group $X$. Most often, the latter can be seen as a uniformly locally bounded space in various natural ways. Our first task is therefore to show that, as far as we are concerned with topologising $\Aut(X)$, most of these ways yield the same result. Nonetheless, there are still two natural topologies to put on $\Aut(X)$, that we name the \emph{upper} and the \emph{lower} topologies (Definition~\ref{def:topuplow}). These topologies agree for a large class of groups, for instance, for all SIN groups and for all locally compact groups (Proposition~\ref{prop:csinlowup}). When they do not, the latter seems more interesting. In particular, the conjugation morphism from $X$ to $\Aut(X)$ is always continuous for the lower topology (Proposition~\ref{prop:conjcontlower}). Moreover, this topology makes $\Aut(X)$ a Polish group when $X$ is a Polish locally Roelcke-precompact group (Corollary~\ref{cor:polautlocroelcke}). 

We conclude this article by a discussion of some of the many natural questions that this machinery raises, regarding notably the structure of the automorphism group and the closure of the subgroup of inner automorphisms of a topological group.
 
\subsubsection*{Acknowledgements}

I am grateful to Nicolas Monod for uplifting discussions about this project. It is also a pleasure to thank Arno Fehm and Nicolas Garrel for insightful comments on Galois theory.

\tableofcontents

\section{Uniformly locally bounded spaces}

This Section recalls many well-known notions about uniform spaces and bornologies, as well as several illustrative examples. An experienced reader can directly jump to Section~\ref{sec:ulbspace} for the definition of uniformly locally bounded spaces, and then to Section~\ref{sec:autulb} for the first results about these spaces.

\subsection{Uniform spaces and bornologies}

Let $X$ be a set. For subsets $E, F \subseteq X \times X$, we write
\begin{align*}
E\inv &= \left\{(x,y)\ |\ (y, x) \in E\right\}, \\
E \circ F &= \left\{(x, y)\ |\ \exists z \in X:\ (x, z) \in E,\ (z, y) \in F\right\},
\end{align*}
and, for $x \in X$ and $A \subseteq X$,
\begin{align*}
E[x] &= \left\{y\ |\ (x, y) \in E\right\}, \\
	E[A] &= \bigcup_{x \in A} E[x].
\end{align*}
We also denote by $\Diag$ or $\Diag_X$ the \emph{diagonal} of $X \times X$, namely the set of all points $(x, x)$.

We recall that a \emph{uniform structure} on a set $X$ is a collection $\Unif$ of subsets of $X \times X$ such that:
\begin{enumerate}[label=\textbf{(U\arabic*)}]
\item $X \times X \in \Unif$;
\item $\Diag \subseteq E$ for any $E \in \Unif$;
\item $E\inv \in \Unif$ for any $E \in \Unif$;
\item if $E \subseteq F$ and $E \in \Unif$, then $F \in \Unif$;
\item if $E, F \in \Unif$, then $E \cap F \in \Unif$;
\item for any $E \in \Unif$, there is $F \in \Unif$ such that $F \circ F \subseteq E$.
\end{enumerate}

Elements of $\Unif$ are called \emph{entourages} or \emph{uniform entourages}. A \emph{basis} or \emph{fundamental system of entourages} for $\Unif$ is any subset $\Subsetunif$ of $\Unif$ such that any uniform entourage contains an element of $\Subsetunif$. A uniform structure $\Unif$ is called \emph{Hausdorff} if for any $x \neq y \in X$, there is an entourage $E$ such that $(x, y) \not\in E$---or, in short: $\bigcap \Unif = \Diag_X$.

A uniform structure gives rise to a topology on $X$, the one for which $E[x]$ is a neighborhood of $x$ for any $x \in X$ and $E \in \Unif$. We will always endow by default a uniform space with the topology defined by its uniform structure. Observe that this topology is Hausdorff if and only if the uniform structure is so.

Recall moreover that a net $(x_\alpha)$ in $X$ is called \emph{Cauchy} if for any uniform entourage $E$, there is an index $\alpha_0$ such that $(x_\alpha, x_\beta) \in E$ for any $\alpha, \beta \geq \alpha_0$. A uniform structure is said \emph{complete} if any Cauchy net is convergent (for the topology defined by the uniform structure).

We refer to \cite{Weil_uniforme}, \cite[\textsc{ii}]{Bourbaki_TG_I}, \cite[chap.~8]{Engelking_GT} or any good textbook on topology for more information on uniform structures and we now proceed to the abstract setting for boundedness.

A \emph{bornology} on a set $X$ is a collection $\Born$ of subsets of $X$ such that:
\begin{enumerate}[label=\textbf{(B\arabic*)}]
\item\label{born:singleton} for any $x \in X$, $\{x\} \in \Born$;
\item\label{born:stabsubset} for any $A \in \Born$ and $B \subseteq A$, $B \in \Born$;
\item\label{born:stabconnunion} for any $A, B \in \Born$ \emph{such that} $A \cap B \neq \emptyset$, $A \cup B \in \Born$.
\end{enumerate}
Elements of $\Born$ are called \emph{bounded sets}. A bornology is called \emph{connected} if the axiom~\ref{born:stabconnunion} can be strengthened into:
\begin{enumerate}[label=\textbf{(B3')}]
 \item\label{born:stabunion} for any $A, B \in \Born$, $A \cup B \in \Born$.
\end{enumerate}

A \emph{basis} of a bornology $\Born$ is any subset $\Born' \subseteq \Born$ such that any bounded set is contained in an element of $\Born'$. We say that a bornology is \emph{countably generated} if it admits a countable basis.

\begin{remark}
	We differ slightly from Bourbaki in this definition. For the scholar of Nancago, a bornology satifies \ref{born:stabsubset} and \ref{born:stabunion}, and a \emph{covering} bornology satisfies moreover \ref{born:singleton}. Thus our \enquote{connected bornologies} are Bourbaki's \enquote{covering bornologies}. This difference, without much of a consequence in concrete examples, finds its origin in Proposition~\ref{prop:dictcoarsborn} below.
\end{remark}

Let us quickly review examples of these structures
\begin{description}
	\item[Metric spaces] Let $(X, \dist)$ be a pseudometric space and define for any $\alpha > 0$, the set $E_\alpha = \left\{(x, y)\ |\ \dist(x, y) < \alpha\right\}$. Then, as easily seen, the collection of supersets of all $E_\alpha$ is a uniform structure on $X$, which is Hausdorff if and only if $\dist$ is a metric. A subset of $X$ is said \emph{bounded} if it has finite diameter; the bounded sets form a connected bornology\footnote{If the pseudometric were allowed to take the value $\infty$---that is, is an \emph{écart}---then the bornology would not be connected.}.
	\item[Topological vector spaces] Let $X$ be a topological vector space. For any subset $A$ of $X$, we define an entourage $E_V = \left\{(x, y)\ |\ x - y \in V\right\}$. Then the collection of supersets of all $E_V$, where $V$ runs among all neighborhoods of $0$, is a uniform structure on $X$, which is Hausdorff if and only if $V$ is so. A subset of a topological vector space is called \emph{bounded} if it is absorbed by any neighborhood of $0$. The collection of all bounded sets is a connected bornology.
	\item[Topological groups] Let $X$ be a topological group. The \emph{left} uniform structure on $X$ is generated by the sets $V^\Left = \left\{(x, y)\ |\ x\inv y \in V\right\}$ when $V$ runs among all identity neighborhoods of $G$. Once again, it is Hausdorff if and only if $G$ is so. A subset $B$ of a topological group $G$ is said \emph{bounded} if for any increasing exhaustive sequence of open sets $(V_n)$ such that $V_n^2 \subseteq V_{n+1}$, there is an index $k$ such that $B \subseteq V_k$. We will review other uniform structures on topological groups in Section~\ref{sec:unifstrgroup}.
\item[Trivial structures] On any set $X$, the collection of all supersets of the diagonal is a uniform structure, called the \emph{discrete uniform structure}. It induces the discrete topology on $X$. Similarly, the collection of all subsets of $X$ is a connected bornology, called the \emph{trivial bornology}.
\item[Fine uniform structure] Whenever the topology of a space is definable via a uniform structure\footnote{Which happens exactly when points can be separated from disjoint closed subsets by continuous functions, see~\bbkitgix{1}{5}{th.~2}.}, there exists a largest (for the inclusion in $\Power(X \times X)$) uniform structure that induces the given topology \cite[\textsc{ix}, \S~1, exerc.~10]{Bourbaki_TG_V}. This structure is called the \emph{fine uniform structure}. Any continuous function from $X$ to a uniform space $Y$ is automatically uniformly continuous when $X$ is endowed with the fine uniform structure.
\item[Finite and compact bornologies] On any set $X$, the collection of all finite subsets of $X$ is a connected bornology. More generally, on any Hausdorff topological space $X$, the collection of all relatively compact subsets of $X$ is a connected bornology.
\end{description}

Lastly, let us recall the usual terminology for some kind of morphisms of these structures. A map $f\colon X \rightarrow Y$ between two uniform spaces is said \emph{uniformly continuous} if $(f \times f)\inv (E)$ is a uniform entourage of $X$ for any uniform entourage $E$ of $Y$. If $\Coll$ is a collection of subsets of $X$, we say that $f$ is \emph{uniformly continuous on $\Coll$} if for any $A \in \Coll$, the restriction of $f$ to $A$ is uniformly continuous. (A subset $A$ of a uniform space $X$ is always assumed to be endowed with the induced uniform structure, that is, with the trace on $A \times A$ of the uniform entourages of $X$.) Finally, a map between two spaces endowed with some bornology is said \emph{modest} if it sends bounded sets to bounded sets. 

\subsection{Lagniappe: coarse spaces}\label{sec:coarse}

This subsection explains the notion of coarse spaces. Stricto sensu, it is not needed for our study and can thus be safely skipped on a first reading. However, it has natural links with both notions of uniform structure and bornology and can thus help to understand some tangential remarks.

A \emph{coarse structure} is a collection $\Coar$ of subsets of $X \times X$ such that:
\begin{enumerate}[label=\textbf{(C\arabic*)}]
\item $\Diag \in \Coar$;
\item $E\inv \in \Coar$ for any $E \in \Coar$;
\item if $E \subseteq F$ and $F \in \Coar$, then $E \in \Coar$;
\item if $E, F \in \Coar$, then $E \cup F \in \Coar$;
\item\label{coar:stabcircle} if $E, F \in \Coar$, then $E \circ F \in \Coar$.
\end{enumerate}
Elements of $\Coar$ are called \emph{entourages} or \emph{coarse entourages}. We will of course avoid the former short name if a uniform structure is also around. We refer to \cite{Roe_coarse} for more informations on coarse structures.

A coarse structure $\Coar$ is said \emph{(coarsely) connected} if for any $x, y \in X$, we have $\{(x, y)\} \in \Coar$---or, in short: $\bigcup \Coar = X \times X$. Observe that this axiom is symmetric to the Hausdorff axiom for uniform structures.

A coarse structure gives rise to a bornology by declaring  $A \subseteq X$ to be \emph{bounded} if $A \times A \in \Coar$. On the other hand, for a given bornology $\Born$, there exists a finest coarse structure (that is, a minimal one for the relation of inclusion in $\Power(X \times X)$) whose collection of bounded sets is $\Born$. These phenomena, similar to the topology induced by a uniform structure and to the finest uniform structure of a completely regular space, are summarised in the following dictionary:

\begin{proposition}\label{prop:dictcoarsborn}
Let $X$ be a set.
\begin{itemize}
\item Let $\Coar$ be a coarse structure on $X$. The collection $\Born$ of bounded sets is a bornology. Moreover, if $\Coar$ is connected, then $\Born$ is a connected bornology.
\item Conversely, if $\Born$ is a bornology, then the collection of all subsets of sets of the form
\begin{equation}\label{eq:coarsentborn}
	\Diag_X \cup \bigcup_{i = 1}^n A_i \times A_i, \tag{\ddag}
\end{equation}
where $A_i \in \Born$, is a coarse structure $\Coar_\Born$ whose collection of bounded sets is $\Born$ and that is contained in any other coarse structure for which all sets in $\Born$ are bounded. Moreover, if $\Born$ is connected, then $\Coar_\Born$ is connected.
\end{itemize}
\end{proposition}

\begin{proof}
\begin{itemize}
 \item Axioms~\ref{born:singleton} and \ref{born:stabsubset} are obvious. For \ref{born:stabconnunion}, observe that, if $A \cap B \neq \emptyset$, then $A \times B = (A \times A) \circ (B \times B)$ and is thus a coarse entourage of $\Coar$. Therefore,
\begin{equation*}
(A \cup B) \times (A \cup B) = (A \times A) \cup (B \times B) \cup (A \times B) \cup (B \times A)
\end{equation*}
is also a coarse entourage, hence $A \cup B$ is bounded. For Axiom~\ref{born:stabunion}, observe that for any nonempty $A, B \in \Born$ and any $a \in A, b \in B$, we have 
\begin{equation*}
A \times B = (A \times A) \circ \{(a, b)\} \circ (B \times B),
\end{equation*}
hence $A \times B$ is also a coarse entourage if $\Coar$ is connected.
\item The only axiom of coarse structures which is not completely straightforward for $\Coar_\Born$ is \ref{coar:stabcircle}. So let 
\begin{equation*}
E = \Diag_X \cup \bigcup_{i = 1}^n A_i \times A_i, \qquad F = \Diag_X \cup \bigcup_{j = 1}^m B_j \times B_j,
\end{equation*}
with $A_i, B_j \in \Born$ and let us show that $E \circ F$ is a subset of a set of the form~\eqref{eq:coarsentborn}. We have
\begin{equation*}
E \circ F = \Diag_X \cup \bigcup_{i = 1}^n A_i \times A_i \cup \bigcup_{j = 1}^m B_j \times B_j \cup G,
\end{equation*}
where $G$ is the union of the sets $A_i \times B_j$ for which $A_i \cap B_j \neq \emptyset$. Since $A_i \times B_j$ is a subset of $(A_i \cup B_j)^2$, Axiom~\ref{born:stabconnunion} ensures that $G$, hence $E \circ F$, is indeed a subset of a coarse entourage of the form~\eqref{eq:coarsentborn}.

It is obvious that all sets in $\Born$ are bounded for $\Coar_\Born$ and that $\Coar_\Born$ is contained in any coarse structure with this property. Let us then show that all bounded sets for $\Coar_\Born$ are in $\Born$. So let $B \subseteq X$ be such that 
\begin{equation*}
B \times B \subseteq \Diag_X \cup \bigcup_{i = 1}^n A_i \times A_i
\end{equation*}
for some $A_i \in \Born$. If $B$ is empty or is a singleton, then Axioms~\ref{born:singleton} and~\ref{born:stabsubset} ensure anyway that $B \in \Born$. Hence we can assume that $B$ has at least two elements, which means that $B \times B \subseteq \bigcup_{i = 1}^n A_i \times A_i$. Now if $B$ is not in $A_n$, then $B \times B$ is actually contained in $\bigcup_{i = 1}^{n - 1} A_i \times A_i$. An easy induction thus implies that $B$ is a subset of some $A_j$, hence belongs to $\Born$.

Lastly, if $\Born$ is stable by any finite unions, then for any $x, y \in X$, we have $\{x, y\} \in \Born$ by \ref{born:singleton}. Since $(x, y) \in \Diag_X \cup \{x, y\} \times \{x, y\}$, the coarse structure $\Coar_\Born$ is connected.
\end{itemize}
\end{proof}

\begin{remark}
	Despite its name, \enquote{coarse connectedness} should really be thought as an assumption as mild as Hausdorffness. Indeed, the above proposition shows that it is \emph{equivalent} to the fact that the union of two bounded sets is still bounded. Moreover, in the same way as a uniform space admits a canonical \enquote{Hausdorffisation} \cite[\textsc{ii}, \S~3, n°~8]{Bourbaki_TG_I}, any coarse space can be decomposed as a disjoint union of coarsely connected subspaces, since the relation \enquote{$x \sim y$ if $\{(x, y)\} \in \Coar$} is an equivalence relation.
\end{remark}

\begin{remark}
Besides the general construction of the above dictionary, another way to produce a coarse structure compatible with a given bornology is to use some transitive action to move around pairs of points. For instance, on a topological group $G$, the entourages $B^\Left$, defined similarly to their uniform counterparts but with $B$ now running among bounded sets, generate a coarse structure on $G$, which is useful to study its large-scale geometry \cite{Rosendal_coarse}.
\end{remark}

\subsection{Uniformly locally bounded spaces}\label{sec:ulbspace}

\epigraph{[...] ces espaces devraient être dits \emph{uniformément localement compacts} si cette manière de parler n’offensait l’euphonie et même la grammaire\footnotemark.}{A.~Weil, \textit{Sur les espaces à structure uniforme et la topologie générale}, p.~27 (reprinted as p.~171 of \cite{Weil_collected_1}).}
\footnotetext{\emph{[...] these spaces should be called \emph{uniformly locally compact} if this way of speaking were not an offence to euphony and even to grammar.}}

We focus on this note on spaces that simultanuously carry a uniform structure and a bornology in some compatible way. More precisely, we say that a set $X$ is a \emph{uniformly locally bounded space} if it is endowed with a uniform structure and a bornology such that there exists a uniform entourage $E$ for which $E[B]$ is bounded for any bounded set $B$. We will call such an entourage a \emph{bounded uniform entourage}\footnote{Which is a slight abuse of terminology, since the uniform entourage $E$ is usually not a bounded subset of $X \times X$ endowed with the product bornology. Nonetheless, no confusion should arise since we will not consider any bornology on $X \times X$.}. Some comments are in order:
\begin{itemize}
 \item Any uniform entourage $F$ contained in a bounded uniform entourage $E$ is itself bounded.
 \item Since singletons are always bounded, any point admits in particular a bounded neighborhood.
 \item Assume that $X$ is endowed with a coarse structure (cf.~Section~\ref{sec:coarse}). Whenever $E$ is a \emph{coarse} entourage and $B$ is a bounded set, the set $E[B]$ is bounded as well (since $E[B] \times E[B] = E\inv \circ (B \times B) \circ E$). In particular, if a space is endowed with a uniform and a coarse structures that have a nontrivial intersection, then the space is uniformly locally bounded (for the bornology associated to the coarse structure).
 \item If $X$ is uniformly locally bounded, then it stays uniformly locally bounded for any finer uniform structure, with respect to the same bornology.
\end{itemize}

Let us now review the above examples:
\begin{description}
\item[Metric spaces] The diameter of $E_\alpha [A]$ is bounded by $2\alpha + \diam A$. Hence any (pseudo)metric space is uniformly locally bounded (for its standard uniform structure and bornology).
\item[Topological vector spaces] For a topological vector space $X$ to be uniformly locally bounded, we need to find a neighborhood $V$ of $0$ such that $V + B$ is bounded for any bounded set $B$. This happens precisely when $V$ itself is bounded, since the sum of two bounded sets in $X$ is bounded. For locally convex spaces, the existence of a bounded identitiy neighborhood is equivalent to $X$ being a (semi)normed space \cite[\textsc{iii}, \S~1, n°~2, rem.~1]{Bourbaki_EVT}. (See \cite[Th.~3]{Hyers_1938} for the analoguous characterization for general topological vector spaces.)
\item[Topological groups] By the same argument as for topological vector spaces, a topological group is uniformly locally bounded for the left-uniform structure and the standard bornology if and only if it is \emph{locally bounded} (that is, admits a bounded neighborhood of the identity). This is in particular the case of locally compact groups. We will focus on locally bounded groups in Section~\ref{sec:groups}.
\item[Trivial examples] Any uniformly discrete space is uniformly locally bounded with respect to any bornology. Any uniform space is uniformly locally bounded with respect to the trivial bornology.
\item[Locally compact spaces] Uniformly locally bounded spaces with respect to the compact bornology (that is, uniform spaces admitting an entourage $E$ such that $E[K]$ is relatively compact for all compact subsets $K$\footnote{An easy compactness argument shows that this is the case as soon as $E[x]$ is compact for any point $x$.}) are known as \emph{uniformly locally compact} spaces. Their regularity properties make these spaces palatable for measure-theoretic considerations (see for instance \cite[4.16 and 7.20]{Pachl_2013}). Note that not all locally compact spaces can be considered as uniformly locally compact (see Section~\ref{sec:loccomp}).
\end{description}

\begin{remark}\label{rem:intertwining}
 The intertwining of a uniform structure and a bornology as in a uniformly locally bounded space automatically ensures the following properties:
\begin{enumerate}
\item The closure of any bounded set is bounded. Indeed, in a uniform space $X$, the closure of a subset $A$ is the intersection of all the sets $E[A]$, where $E$ ranges among uniform entourages \bbkitgii{1}{2}{cor.~1}.
\item\label{rem:intertwining:complete} A uniformly locally bounded space is complete if and only if all its closed bounded sets are complete. Indeed, if $E$ is a bounded uniform entourage and $(x_\alpha)$ is a Cauchy net, then there is an index $\alpha_0$ such that $(x_\alpha, x_\beta) \in E$ for any $\alpha, \beta \geq \alpha_0$, hence in particular the net $(x_\alpha)_{\alpha \geq \alpha_0}$ is contained in the closed bounded set $\overline{E[x_{\alpha_0}]}$.
\item If a uniformly locally bounded space is connected as a topological space, then its bornology is also connected. Indeed, for any uniform entourage $E$, the set $E^\infty = \bigcup_n E^{\circ n}$ (where $E^{\circ n}$ is the n-fold composition $E \circ \dots \circ E$) is both open and closed in $X \times X$ (because $E^\infty = E \circ E^\infty \circ E$). In particular, if $X$ is topologically connected, then $X \times X = E^\infty$. Now if $A$ and $B$ are two subsets of $X$ and $x$ is any point, there is thus $n$ such that $E^n [A]$ and $E^n [B]$ contain $x$. If moreover $A$ and $B$ are bounded and $E$ is a bounded uniform entourage, then $A \cup \{x\}$ and $B \cup \{x\}$ are two bounded sets with a non-empty intersection, and their union contains $A \cup B$. (The converse does not hold, since for instance the bornology of any metric space or topological group is connected.)
\item\label{rem:intertwining:sigma} The above argument shows also that a (topologically) connected uniformly locally bounded space is the union of countably many of its bounded sets.
\end{enumerate}
\end{remark}

\subsection{Uniform structures on groups}\label{sec:unifstrgroup}

Before proceeding further to the study of uniformly locally bounded spaces, let us recall some facts and terminology about uniform structures on groups. 

Several natural compatible uniform structures can be defined on a topological group $X$ in such a way that translations become uniformly continuous (see \cite{RD_1981} for a comprehensive study), such as:
\begin{itemize}
\item The \emph{left uniform structure} $\UnifLeft$ is generated by entourages of the form $V^\Left = \{(x, y)\ |\ x\inv y \in V\}$, where $V$ runs among identity neighborhoods in $X$. Alternatively, it can be defined as the unique unifom structure compatible with the given topology on $X$ and generated by left-invariant entourages (where an entourage $E$ is said \emph{left-invariant} if $(x, y) \in E$ implies $(gx, gy) \in E$ for all $g \in X$).
\item The \emph{right uniform structure} $\UnifRight$ is the mirror analogue of the the structure $\Left$; it is generated by entourages of the form $V^\Right = \{(x, y)\ |\ x y\inv \in V\}$.
\item The \emph{upper uniform structure} $\UnifUpper$ is the supremum of the left and right uniform structures. It is thus generated by entourages of the form $V^\Upper = V^\Left \cap V^\Right = \ens{(x, y)}{xy\inv, x\inv y \in V}$.
\item The \emph{lower} or \emph{Roelcke uniform structure} $\UnifRoelcke$ is the infimum of the left and right uniform structures. It is generated by entourages of the form $V^\Roelcke = \ens{(x, y)}{x \in VyV}$ \cite[2.5]{RD_1981}.
\end{itemize}

In general, these four structures are distinct. A group is called \emph{SIN} if these four structures are the same. It is equivalent to requiring the existence of a basis of identity neighborhoods that are invariant by conjugation \cite[2.17]{RD_1981}.

It follows from their definition that these structures admit a countable basis of uniform entourages---that is, are (pseudo)metrisable---as soon as the topological group $X$ admits a countable basis of identity neighborhoods, which in turns is equivalent to the existence of a \emph{left-invariant} compatible metric (see e.g.~\cite[Th.~2.B.2]{CH_2016}).

Observe that the inverse map $x \mapsto x\inv$ is an automorphism of $\UnifUpper$ and of $\UnifRoelcke$, and is also an isomorphism between $\UnifLeft$ and $\UnifRight$. Therefore, there are three notions of completeness for a topological group:
\begin{itemize}
\item being complete for the upper uniform structure or \emph{Ra\u{\i}kov-complete};
\item being complete for the left or equivalently for the right uniform structure, or \emph{Weil-complete}\footnote{This is also often simply described as \enquote{being a complete group}.};
\item being complete for the lower uniform structure or \emph{Roelcke-complete}.
\end{itemize}

A Roelcke-complete group is in particular Weil-complete, and a Weil-complete group is Ra\u{\i}kov-complete. In general, these implications cannot be reversed \cite[8.12--8.14]{RD_1981}; however all locally compact groups are Roelcke-complete \cite[8.8]{RD_1981}.

Recall that a topological space is called \emph{Polish} if its topology is separable and can be defined via some complete metric. For topological \emph{groups}, this metric can be chosen to be one defining the upper uniform structure, that is: a topological group is Polish if and only if it is separable, metrisable, and Ra\u{\i}kov-complete (see for instance \cite[1.2.2]{BK_1996}).

\section{Automorphism groups of uniformly locally bounded spaces}\label{sec:autulb}

We define a \emph{morphism of uniformly locally bounded spaces} as a map that is uniformly continuous on bounded sets and that sends bounded sets to bounded sets. In particular, an \emph{automorphism} of a uniformly locally bounded space is a bijection $\phi$ such that, for any bounded set $B$, both $\phi(B)$ and $\phi\inv (B)$ are bounded and, moreover, $\phi$ is a uniform isomorphism between $B$ and $\phi(B)$. We write $\Autulb(X)$ for the group of all automorphisms of a uniformly locally bounded space $X$.

The first Subsection is devoted to the topology to be put on $\Autulb(X)$. The second one quickly reviews some topological properties of $\Autulb(X)$ that can be deduced from $X$. The last Subsection deals with completeness.

\subsection{Topology of uniform (bi)convergence on bounded sets}\label{sec:topcomp}

Let $\Coll$ be a non-empty collection of subsets of a set $X$ and $Y$ be a uniform space. For a set $A \in \Coll$ and a uniform entourage $E$ on $Y$, we define an entourage $\Entsc_{A, E}$ on the space $X^Y$ of maps from $Y$ to $X$ as
\begin{equation*}
	\Entsc_{A, E} = \ens{(f, g)}{\forall x \in A,\ (f(x), g(x)) \in E}.
\end{equation*}
The finite intersections of all sets $\Entsc_{A, E}$ generate a uniform structure on the set $X^Y$, called the \emph{uniform structure of $\Coll$-convergence}. The topology it induces is called the \emph{topology of $\Coll$-convergence}, or the \emph{topology of uniform convergence on elements of $\Coll$}. Concretely, a net of maps $f_i$ converges to a map $f$ for the topology of $\Coll$-convergence if for any set $A \in \Coll$ and any uniform entourage $E$ on $Y$, there is an index $j$ such that $(f_i(x), f(x)) \in E$ for any $x \in A$ and $i \geq j$.

We will consider the special case where $X = Y$ is a uniformly locally bounded space and $\Coll$ is its bornology. On any subgroup $G$ of $\Sym(X)$, we define the \emph{uniform structure of $\Coll$-biconvergence} as the weakest uniform structure on $G$ such that $g \mapsto g$ and $g \mapsto g\inv$ are both uniformly continuous as maps from $G$ to $X^X$ endowed with the uniform structure of $\Coll$-convergence. The topology it defines will be called the \emph{topology of uniform biconvergence on bounded sets}; it can also directly be defined as the weakest topology on $G$  such that $g \mapsto g$ and $g \mapsto g\inv$ are both continuous as maps from $G$ to $X^X$ endowed with the topology of $\Coll$-convergence \bbkitgii{2}{3}{cor.}. In other words, the net $(g_i)$ converges to $g$ for that topology if both $g_i \cvg g$ and $g_i\inv \cvg g\inv$ for the topology of uniform convergence on bounded sets.

\begin{remark}
If the bornology is connected (that is, closed by finite unions), then a fundamental system of uniform entourages of the uniform structure of $\Coll$-convergence is already given by the sets $\Entsc_{A, E}$ (without need to consider finite intersections of the latter). In that case, a fundamental system of uniform entourages for the structure of $\Coll$-\emph{biconvergence} is simply given by sets of the form
\begin{equation*}
	\ens{(f, g)}{\forall x \in A,\ (f(x), g(x)) \in E\ \text{and}\ (f\inv (x), g\inv (x)) \in E}.
\end{equation*}
\end{remark}

We can now state the core technical result of this note:

\begin{theorem}\label{th:comp}
Let $X$ be a uniformly locally bounded space and let $\Endo$ be its semigroup of endomorphisms. The composition map $\Endo \times \Endo \rightarrow \Endo$ is continuous when $\Endo$ is endowed with the topology of uniform convergence on bounded sets. In particular, $\Autulb(X)$ is a topological group for the topology of uniform biconvergence on bounded sets.
\end{theorem}

\begin{proof}
	Let $(g_i)$ and $(h_i)$ be two nets in $\Endo$ converging uniformly on bounded sets, respectively to $g$ and $h$. Let $B$ be a bounded set in $X$ and $E$ a uniform entourage.

	Since $X$ is uniformly locally bounded, we can find a bounded uniform entourage $F'$ such that $F' \circ F' \subseteq E$. Since $h$ is modest and $F'$ is bounded, the set $B' = F'[h(B)]$ is bounded. Therefore, since $g$ is uniformly continuous on bounded sets, we can find another uniform entourage $F \subseteq F'$ such that $(g \times g)(F \cap (B' \times B')) \subseteq F'$.

	As $(h_i)$ converges to $h$, there is an index $j$ such that $(h_i (x), h (x)) \in F$ for any $x \in B$ and $i \geq j$. In particular, $h_i (B) \subseteq F[h(B)] \subseteq B'$, hence we also have $(gh_i (x), gh(x)) \in F'$ by our choice of $F$.

 On the other hand, we can find an index $j' \geq j$ such that $(g_i (y), g (y)) \in F'$ for any $i \geq j'$ and $y \in B'$, therefore in particular for $y = h_i (x)$ and $x \in B$.

	Thence, for any $x \in B$ and $i \geq j'$, we have
	\begin{equation*}
		(g_i h_i (x), gh (x)) \in F' \circ F' \subseteq E,
	\end{equation*}
	which shows that $(g_i h_i)$ converges to $gh$.
\end{proof}

\begin{example}[Trivial cases]\label{ex:trivial}
If $X$ is endowed with the discrete uniform structure and the finite bornology, $\Autulb(X)$ is nothing but $\Sym(X)$ with the pointwise convergence topology. (For examples involving other bornologies, see Sections~\ref{sec:exotic} and \ref{sec:galois} below.) If $X$ is endowed with any uniform structure and the trivial bornology, $\Autulb(X)$ is the group of all automorphisms of the uniform structure, with the uniform convergence topology.
\end{example}

\begin{example}
In many natural examples of uniformly locally bounded spaces, the bornology is actually defined \emph{via} the uniform structure (or the topology it defines). In particular, it is easy to see that many continuous maps of these spaces are automatically modest:
	\begin{enumerate}
		\item Lipschitz (or even Hölder) maps between metric spaces are modest.
		\item Continuous linear maps between topological vector spaces are modest \cite[\textsc{iii}, \S~1, n°~3]{Bourbaki_EVT}.
		\item Continuous homomorphisms between topological groups are modest \cite[2.35]{Rosendal_coarse}.
		\item For the compact bornology, any homeomorphism is modest. 
	\end{enumerate}
	Therefore, although the \emph{full} group $\Autulb(X)$ might be hard to handle, we can often nonetheless appeal to Theorem~\ref{th:comp} for some more common groups that can be realized as subgroups of some $\Autulb(X)$. By the above examples, this is notably the case:
\begin{enumerate}
\item of the isometry group of metric spaces (see Section~\ref{sec:metricspaces});
\item of the general linear groups of normed spaces (yielding the norm-operator topology);
\item of the automorphism group of locally bounded groups (this example will occupy us for the whole Section~\ref{sec:groups});
\item of the homeomorphism group of uniformly locally compact spaces (for locally compact spaces in general, see Section~\ref{sec:loccomp}), for which we recover classical results about the compact-open topology.
\end{enumerate}
\end{example}

\begin{remark}
	In \cite{BC_1946}, Braconnier and J.~Colmez announced a result similar to Theorem~\ref{th:comp}, in which the existence of a bounded uniform entourage is replaced by the weaker condition: for any bounded set $A$, there is a uniform entourage $E$ such that $E[A]$ is bounded. The main advantage of this \enquote{local} definition is to encompass immediately all locally compact spaces (with their fine uniform structure and their compact bornology), which seems to have been the goal of Braconnier and Colmez. (We have been unable to locate the proofs of the results announced in \cite{BC_1946}, or indeed any other joint work by these authors, and therefore we can but speculate about their motivations. Their framework seems to be a bit oversized if only locally compact spaces were to be considered.)

	For our part, we favoured the simplicity of the \enquote{global} definition of uniformly locally bounded spaces, which seems more natural in view of the examples coming from metric spaces, topological vector spaces, and topological groups.\footnote{Moreover, it is unclear to us whether their framework allows the general completeness result of Section~\ref{sec:completeness}.} At any rate, although not all locally compact spaces can be seen as uniformly locally bounded spaces, it is easy to apply to these spaces all the results of this note via Alexandrov's one-point compactification, see Section~\ref{sec:loccomp}.
\end{remark}

\begin{remark}
	Maissen proved \cite[Satz 2]{Maissen_1963} that if the semigroup $\Lin(V)$ of continuous linear endomorphisms of a locally convex space $V$ is endowed with some topology of $\Coll$-convergence such that the composition is continuous, then $V$ is a normed space (and therefore $\Coll$ is the collection of all bounded sets). It is thus hopeless to try and extend Theorem~\ref{th:comp} beyond uniformly locally bounded spaces, at least via the topology of $\Coll$-convergence. (Of course, stricto sensu, Maissen's result does not rule out the existence of other compatible group topologies on $\GL(V)$ for some specific non-normable space $V$, but they ought to be very exotic.)
\end{remark}

\begin{remark}[On biconvergence]
 It may happen that the topology of $\Coll$-biconvergence on $\Autulb(X)$ agrees with the topology of $\Coll$-convergence, i.e., that the inversion map is already continuous for the latter. This phenomen occurs for instance with the trivial cases of Example~\ref{ex:trivial}. More interestingly, Arens proved that this is also the case for \emph{locally connected} locally compact spaces \cite[Th.~4]{Arens_1946}---see also \cite[\textsc{x}, \S~3, exerc.~15]{Bourbaki_TG_V} for a generalisation. In general however, the inversion map is not continuous for the topology of $\Coll$-convergence, see for instance \cite[p.~601]{Arens_1946}.
\end{remark}

\begin{remark}\label{rem:contact}
	Although we will not use this fact, it is worth noting that, since points in $X$ admit bounded neighborhoods, the action map $\Autulb(X) \times X \rightarrow X\colon (g, x) \mapsto g(x)$ is continuous (cf.~\bbkitgx{1}{6}{prop.~9}).
\end{remark}

\subsection{Topological properties}\label{sec:topprop}

We collect here a few easy topological facts about $\Autulb(X)$ and its subgroups. The results of this subsection come without surprise, a hasty reader can directly jump to subsection~\ref{sec:completeness} where less obvious phenomena are studied.

\begin{proposition}
 Let $X$ be a uniformly locally bounded space. The group $\Autulb(X)$ is Hausdorff if and only if $X$ is so.
\end{proposition}

\begin{proof}
The intersection of all entourages of a uniform structure of $\Coll$-convergence is contained in the set
\begin{equation}
\ens{(f,g)}{\forall x \in \bigcup \Coll,\ (f(x), g(x)) \in \bigcap \Unif}.
\end{equation}
If $\Coll$ is a bornology, then $\bigcup \Coll = X$ and if $X$ is Hausdorff, then $\bigcap \Unif = \Diag_X$. Therefore, the above set is the diagonal of $X^X$ if we consider the uniform structure of convergence on bounded sets on a Hausdorff uniform space. A fortiori, $\Autulb(X)$ is Hausdorff for the topology of uniform biconvergence on bounded sets.

Conversely, if $X$ is not Hausdorff, then there exist two distinct points $x, y$ such that $(x, y)$ (and also $(y, x)$) belongs to any uniform entourage. In particular, for any $z \in X$ and any uniform entourage $E$, if $(y, z)$ belongs to $E$, then $(x, z)$ belongs to $E \circ E$. It follows that the involution $\sigma$ that swaps $x$ and $y$ and fixes every other point is uniformly continuous. Moreover, $\sigma (B) \subseteq E[B]$ for any bounded set $B$ (since $E$ contains $(x, y)$ and $(y, x)$), hence is modest. Therefore, $\sigma$ belongs to $\Autulb(X)$. By construction, it also belongs to any identity neighborhood, hence $\Autulb(X)$ is not Hausdorff.
\end{proof}

Recall that a topological group is metrisable as soon as it is Hausdorff and first-countable---this is also an instance of the fact that a Hausdorff uniform structure is metrisable if and only if it admits a countable basis of entourages \bbkitgix{1}{4}{prop.~2}. Since a basis of identity neighborhoods for the topology of uniform (bi)convergence on bounded sets can be parametrized by a generating set of the bornology and a basis of uniform entourages, we immeditaley get:
\begin{proposition}\label{prop:autulbmetrisable}
 Let $X$ be a uniformly locally bounded space. If the uniform structure is metrisable and if the bornology is countably generated, then $\Autulb(X)$ is metrisable.
\end{proposition}

In full generality, $\Autulb(X)$ is unlikely to be SIN. However, some subgroups can:

\begin{proposition}\label{prop:sin}
 Let $X$ be uniformly locally bounded space and let $G$ be a subgroup of $\Autulb(X)$. If the uniform structure and the bornology both admit a basis of $G$-invariant sets, then $G$ is SIN.
\end{proposition}

(Observe that $G$ is in particular uniformly equicontinuous.)

\begin{proof}
If $A \subseteq X$, $E \subseteq X \times X$ and $g$ is any bijection of $X$, we have the inclusion
\begin{equation*}
g \Entsc_{A, E} g\inv \subseteq \Entsc_{gA, gE}
\end{equation*}
(where the action of $g$ on $X \times X$ is the diagonal one). Hence if $G$ is a subgroup of $\Autulb(X)$ that preserves both $A$ and $E$, the set $G \cap \Entsc_{A, E} [\id_X]$ is an invariant identitiy neighborhood of $G$.
\end{proof}

Of course, whether $\Autulb(X)$ is large depends on the \enquote{homogeneity} of $X$, hence is hard to settle in general. Let us just record this easy observation:

\begin{proposition}
 Let $X$ be a uniformly locally bounded space. Assume that the group $\Autulb(X)$ is separable and that there exists a separable subset $Y$ of $X$ such that $X = \Autulb(X) \cdot Y$. Then $X$ is separable.
\end{proposition}

\begin{proof}
 Let $A$ and $B$ be countable dense subsets of $\Autulb(X)$ and $Y$, respectively. We claim that the countable set $A \cdot B$ is dense in $X$. Indeed, for any $x \in X$, we can find $g \in \Autulb(X)$ and $y \in Y$ such that $x = g y$. By assumption, there exist nets $(g_i)$ in $A$ and $(y_j)$ in $B$ that converge to $g$ and $y$, respectively. By Remark~\ref{rem:contact}, the net $(g_i \cdot y_j)$ converges to $gy$.
\end{proof}

\begin{remark}
However, separability of $X$ is far from being enough to ensure separability of $\Autulb(X)$, even if the latter acts transitively. An easy counterexample is given by a countably infinite uniformly discrete set $X$ endowed with the trivial bornology. The group $\Autulb(X)$ is then nothing but $\Sym(X)$ with the discrete topology, hence is not separable. For more interesting counterexamples, see Example~\ref{ex:autnonsep} and Sections~\ref{sec:exotic} and~\ref{sec:loccomp} below.
\end{remark}

Lastly, let us have a look at disconnectedness properties. Recall that a uniform structure is called \emph{non-Archimedean} if it admits a basis of uniform entourages satisfying $E \circ E = E$. Observe that an entourage $E$ such that $E \circ E = E$ is both open and closed since, for any uniform entourage $E$, the entourage $E \circ E \circ E$ is always a neighborhood of the \emph{closure} of $E$ \bbkitgii{1}{2}{prop.~2}. A non-Archimedean uniform space is thus always \emph{$0$-dimensional} (that is, its topology admits a basis of clopen sets) and in particular totally disconnected. This property is inherited by its group of automorphisms, and more generally, if $X$ is uniformly locally bounded, by its group $\Autulb(X)$:

\begin{proposition}
Let $X$ be a non-Archimedean uniformly locally bounded space. Then $\Autulb(X)$ is non-Archimedean for the uniform structure of biconvergence on bounded sets. In particular, $\Autulb(X)$ is a $0$-dimensional group.
\end{proposition}

\begin{proof}
Let $E, F$ be uniform entourages of $X$ and $A, B$ be bounded sets. We have
\begin{equation*}
\Entsc_{A, E} \circ \Entsc_{B, F} \subseteq \Entsc_{A \cap B, E \circ F}.
\end{equation*}
In particular, if $E \circ E = E$, then $\Entsc_{A, E}^2 = \Entsc_{A, E}$. Hence if $X$ is non-Archimedean, then so is $X^X$ for the uniform structure of convergence on bounded sets, hence the result.
\end{proof}

\begin{remark}\label{rem:nonarchgroup}
The terminology \emph{non-Archimedean group} usually means that there exists a basis of identity neighbhorhoods made of open subgroups. (This implies in particular that the group is non-Archimedean for all its four natural uniform structures.) The above proposition does not guarantee that. However, observe that if $B$ is a bounded set such that $E[B] = B$ for some uniform entourage $E$ (hence also for any smaller one) and if moreover $E = E \circ E$, then
\begin{equation*}
\ens{f \in \Autulb(X)}{\forall x \in B,\ (f(x), x) \in E \text{ and } (f\inv (x), x) \in E}
\end{equation*}
is an open subgroup of $\Autulb(X)$. In particular, if $X$ is a non-Archimedean uniformly locally bounded space such that
\begin{quote}
 for each bounded set $B$, there exist a bounded set $C$ and a uniform entourage $E$ such that $B \subseteq C = E[C]$,
\end{quote}
then $\Autulb(X)$ is a non-Archimedean group. This applies for instance:
\begin{itemize}
\item if $X$ is a non-Archimedean uniform space endowed with the trivial bornology;
\item if $X$ is a uniformly discrete space endowed with any bornology;
\item if $X$ is an ultrametric space, endowed with its metric uniform structure and metric bornology.
\end{itemize}
\end{remark}

\subsection{Completeness}\label{sec:completeness}

We now study the completeness of $\Autulb(X)$. Let us recall that a net $(x_\alpha)$ in a uniform space $X$ is called \emph{Cauchy} if for any uniform entourage $E$ of $X$, there is an index $\alpha_0$ such that $(x_\alpha, x_\beta) \in E$ for any $\alpha, \beta \geq \alpha_0$. A uniform space is said \emph{complete} if any Cauchy net is convergent.

We first start by an auxiliary result.

\begin{proposition}\label{prop:critcompl}
	Let $X$ be a \emph{complete} uniformly locally bounded space. Let $(g_\alpha)$ be a net of morphisms from $X$ to $X$. If $(g_\alpha)$ is Cauchy for the uniform structure of convergence on bounded sets, then it converges to a morphism $g$ of $X$.

In particular, the group $\Autulb(X)$ is complete for the uniform structure of biconvergence on bounded sets.
\end{proposition}

\begin{proof}
The argument uses the standard fact that, for a net which is Cauchy with respect to some uniform structure of $\Coll$-convergence to be convergent, it suffices that it admits a pointwise limit (see e.g.~\bbkitgx{1}{5}{prop.~5}).

	Let then $(g_\alpha)$ be a Cauchy net for the uniform structure of convergence on bounded sets. A fortiori, for any $x \in X$, the net $(g_\alpha (x))$ is Cauchy in $X$, hence convergent to some point $g(x)$ by completeness. Therefore, $(g_\alpha)$ converges uniformly on bounded sets to some map $g$ in the space $X^X$.

Since a uniform limit of uniformly continuous maps is itself uniformly continuous, the map $g$ is uniformly continuous on bounded sets. Let us check that it is modest. Let $B$ be a bounded set of $X$ and $E$ a bounded uniform entourage. There exists an index $\alpha_0$ such that $(g_\alpha, g) \in \Entsc_{E,B}$ for any $\alpha \geq \alpha_0$. In particular, $g(B) \subseteq E[g_{\alpha_0} (B)]$, which is bounded since $E$ is bounded and $g_{\alpha_0}$ is modest.

As for the group $\Autulb(X)$, observe that, if $(g_\alpha)$ is Cauchy for the uniform structure of biconvergence on bounded sets, then both $(g_\alpha)$ and $(g_\alpha\inv)$ are Cauchy for the uniform structure of convergence on bounded sets. Therefore, they converge to some morphisms $g$ and $h$ respectively. Since their product converges to $gh$ by Theorem~\ref{th:comp}, we have $g=h\inv$. Hence $(g_\alpha)$ indeed converges to an automorphism of $X$.
\end{proof}

Being a topological group, $\Autulb(X)$ also carries various uniform structures, in general incomparable with the uniform structure of biconvergence on bounded sets. Recall from Section~\ref{sec:unifstrgroup} that a topological group is called \emph{Ra\u{\i}kov-complete} if it is complete for the \emph{upper uniform structure}.

\begin{theorem}\label{th:raikov}
Let $X$ be a complete uniformly locally bounded space. The group $\Autulb(X)$, endowed with the topology of uniform biconvergence on bounded sets, is Ra\u{\i}kov-complete.
\end{theorem}

\begin{proof}
 Let $\UnifUpper$ be the upper uniform structure on $\Autulb(X)$ and write $\UnifConv$ for the uniform structure of biconvergence on bounded sets. A priori, these structures are not comparable, and moreover the former does not admit a basis of entourages closed for the pointwise convergence topology. Therefore, we cannot use the same trick as in the beginning of the proof of Proposition~\ref{prop:critcompl}.

However, $\UnifUpper$ and $\UnifConv$ still induce the same topology on $\Autulb(X)$. The theorem will thus be proved if we show that a Cauchy net for the former is still a Cauchy net for the latter: by Proposition~\ref{prop:critcompl}, it will then be convergent.

The key-point is the fact that a $\UnifUpper$-Cauchy net cannot move bounded sets too far away. More precisely:
\begin{lemma}
If $(g_\alpha)$ is a Cauchy net for $\UnifUpper$, then for any bounded set $B$, there is an index $\alpha_0$ such that
\begin{equation*}
\bigcup_{\alpha \geq \alpha_0} g_\alpha (B)
\end{equation*}
is bounded.
\end{lemma}

\begin{proof}[Proof of the lemma]
Let $E$ be a bounded uniform entourage and $\alpha_0$ be an index such that $(g_\beta\inv \circ g_\alpha) \in \Entsc_{E, B} [\id_X]$ for any $\alpha, \beta \geq \alpha_0$. That is, $((g_\beta\inv \circ g_\alpha) (x), x) \in E$ for any $x \in B$. Therefore, $(g_\beta\inv \circ g_\alpha) (B) \subseteq E[B]$, and in particular $g_\alpha (B) \subseteq g_{\alpha_0} (E[B])$, which is bounded since $E$ is bounded and $g_{\alpha_0}$ modest.
\end{proof}

We can now conclude the proof of Theorem~\ref{th:raikov}. Let $(g_\alpha)$ be a Cauchy net for $\UnifUpper$. Let $B$ be a bounded set and $E$ an uniform entourage of $X$. Choose, thanks to the above lemma, some bounded set $B'$ and an index $\alpha_0$ such that $g_\alpha (B) \subseteq B'$ for any $\alpha \geq \alpha_0$. 

Since $(g_\alpha)$ is $\UnifUpper$-Cauchy, there exists an index $\alpha_1 \geq \alpha_0$ such that, for any $\alpha, \beta \geq \alpha_1$, we have
\begin{equation*}
((g_\alpha \circ g_\beta\inv)(x), x) \in E
\end{equation*}
for any $x \in B'$. Putting $x = g_\beta (y)$ with $y \in B$, we see that
\begin{equation*}
(g_\alpha, g_\beta) \in \Entsc_{E, B}.
\end{equation*}

In particular, $(g_\alpha)$ is Cauchy for the uniform structure of convergence on bounded sets. But the same is true for $(g_\alpha\inv)$, since the latter is also Cauchy for $\UnifUpper$. Hence $(g_\alpha)$ is Cauchy for $\UnifConv$, as required.
\end{proof}

We'll see below (Theorem~\ref{th:raikovlower}) another situation where we can yield completeness on some subgroup of $\Autulb(X)$ \emph{without} completeness on $X$.

\begin{remark}
	Observe that the limit of a Cauchy net of automorphisms of the whole uniform structure and of the bornology may well fail to be uniformly continuous \emph{on the whole space} (hence our definition of morphisms of uniformly locally bounded spaces). For instance, if $X = \Reels$ is endowed with its usual metric uniform structure and bornology, the maps $g_n$ defined by
\begin{equation*}
g_n (x) = \begin{cases}
x^3 &\text{if } x \in [-n, n] \\
n^3 + (x - n) &\text{if } x > n \\
-n^3 + (x + n) &\text{if } x < -n
\end{cases}
\end{equation*}
are automorphisms of the uniform structure and of the bornology, but they converge uniformly on bounded sets to $x \mapsto x^3$, which is not uniformly continuous on $\Reels$.
\end{remark}

\begin{remark}
Completeness of $\Autulb(X)$ with respect to the left or right uniform structure (that is, \emph{Weil-completeness}) should not be expected, even in the simplest cases where $X$ is compact or discrete. Indeed, $\Homeo([0,1])$ is not Weil-complete \cite[p.~603]{Arens_1946}, nor is $\Sym(\Ent)$ (for the topology of pointwise convergence on $\Ent$).
\end{remark}

\section{The case of topological groups}\label{sec:groups}

Of particular interest is the case where the uniformly locally bounded space is a topological group. Recall from Section~\ref{sec:unifstrgroup} that a topological group $X$ carries four natural uniform structures compatible with its topologies, the left, right, upper and lower uniform structures. Recall also that a subset $A$ of a topological group $X$ is said \emph{bounded}\footnote{Or \enquote{coarsely bounded} in Rosendal's terminology.} if, for any exhaustive sequence $V_1 \subseteq V_2 \subseteq \dots$ of open subsets of $X$ such that $V_n^2 \subseteq V_{n+1}$, we have $A \subseteq V_k$ for some index $k$. This bornology is closed by finite unions and inversion \cite[2.10]{Rosendal_coarse}\footnote{There are also four natural coarse structures that produce these bounded sets; they are generated by the same kind of entourages as their uniform counterparts, $B^\Left$, $B^\Right$, $B^\Upper$, and $B^\Roelcke$, but with $B$ now running among bounded sets. See \cite[2.2]{Zielinski_2019} for more details. However, in our study, only the bornology plays a crucial rôle.}. (For considerations about other bornologies on topological groups, see Section~\ref{sec:otherborn}.)

It follows immediately from the definitions that a \emph{continuous morphism} between topological groups is automatically modest and uniformly continuous for any of the four natural uniform structures. In particular, if $X$ is uniformly locally bounded, we can view $\Aut (X)$ (the group of all bicontinuous automorphisms of $X$) as a subgroup of $\Autulb (X)$. Moreover, this subgroup is \emph{closed} for the topology of uniform biconvergence on bounded sets (indeed, a pointwise limit of group morphisms is a group morphism, and all elements of $\Autulb (X)$ are continuous).

Endowing a topological group $X$ with one of the four natural uniform structures yields a priori four different natural notions of \emph{uniformly locally bounded groups}. We will first see (Proposition~\ref{prop:ulblocbound}) that, fortunately, these notions amount to the same condition, namely that there exists a bounded identity neighborhood. We then proceed to show that, when restricted to $\Aut(X)$, the four possible topologies of uniform biconvergence on bounded sets are only two in number, that we call the upper and the lower topologies. The short Section~\ref{sec:conjcont} shows that the conjugation morphism from $X$ to $\Aut(X)$ is continuous for the lower topology. Lastly, Section~\ref{sec:complgroup} establish completeness results for the group $\Aut(X)$, and in particular shows that the automorphism group of a locally Roelcke-precompact group carries a natural Polish topology.

\subsection{Lower and upper topologies}

\begin{proposition}\label{prop:ulblocbound}
 Let $X$ be a topological group, endowed with its standard bornology. Let $\Unif$ be any of the four natural uniform structures on $X$. The following are equivalent:
\begin{itemize}
\item $X$ is uniformly locally bounded when endowed with $\Unif$;
\item $X$ is locally bounded (i.e., admits a bounded identity neighborhood).
\end{itemize}
\end{proposition}

\begin{proof}
Obviously, any point of a uniformly locally bounded spaces admits a bounded neighborhood. Conversely, assume there is a bounded identity neighborhood $V$. For any bounded set $B$, the set $V^\Roelcke [B] = VBV$ is a product of three bounded sets, hence is itself bounded. Therefore, the group is uniformly locally bounded with respect to the Roelcke uniform structure, thus a fortiori with respect to any finer uniform structure.
\end{proof}

 We will thus henceforth consider locally bounded groups. This class of groups is quite wide and has attracted recent attention:
\begin{enumerate}
\item Since compact sets in a group are bounded, any locally compact group is locally bounded. (Moreover, any bounded set in a $\sigma$-compact locally compact group is relatively compact.)
\item More generally, Roelcke-precompact sets are bounded\footnote{A subset $A$ of a group $X$ is said \emph{Roelcke-precompact} if it is precompact for the lower uniform structure, that is: for any identity neighborhood $V$ of $X$, there exists a \emph{finite} set $F$ of $X$ such that $A \subseteq VFV$.}. Locally Roelcke-precompact groups, of which many arise naturally in the context of first-order metric logic, have been studied extensively in \cite{Zielinski_2019}. Besides locally compact groups, we find among them $\Sym(X)$ (and, more generally, pro-oligomorphic groups \cite[2.4]{Tsankov_2012}), the unitary group of a Hilbert space, the automorphism group of the random graph, the automorphism group of the $\aleph_0$-regular tree, the isometry group of the Urysohn space, and so on (see~\cite[\S~4]{Zielinski_2019}).
\item Other examples include the homeomorphism group of spheres or the additive group of normed spaces. See \cite{Rosendal_coarse}, in particular \S~3.2.
\end{enumerate}

When $X$ is a locally bounded group, there are a priori four topologies that can be put on $\Autulb(X)$, corresponding to its four natural uniform structures. However, when restricted to the subgroup $\Aut(X)$, there are only two:

\begin{proposition}\label{prop:lruagree}
Let $X$ be a topological group. On $\Aut(X)$, the uniform structures of uniform convergence on bounded sets with respect to the left, right, or upper uniform structures on $X$ agree with each other.
\end{proposition}

\begin{proof}
Let us write $\UcLeft$, $\UcRight$ and $\UcUpper$ for these three uniform structures. Obviously, $\UcUpper$ is finer than the other two. Moreover, since a basis for the uniform structures $\Left \vee \Right$ on $X$ is given by the intersections of a left uniform entourage and a right one, the uniform structure $\UcUpper$ is also the supremum of the structures $\UcLeft$ and $\UcRight$. Hence we only need to prove that the latter two agree.

Now for any identity neighborhood $V$ in $X$, any bounded set $B$, and any \emph{group automorphisms} $f$ and $g$ of $X$, we have
\begin{align*}
\forall x \in B,\ f(x)\inv g(x) \in V \quad \Leftrightarrow \quad \forall x \in B\inv,\ f(x) g(x)\inv \in V.
\end{align*}
Any bounded set being contained in a symmetric one, this shows that $\UcLeft$ and $\UcRight$ are generated by the same basis of entourages.
\end{proof}

This allows the following definition.

\begin{definition}\label{def:topuplow}
	Let $X$ be locally bounded group. On $\Aut(X)$, the topology of uniform biconvergence on bounded sets with respect to the upper (or left, or right) uniform structure of $X$ will be called the \emph{upper topology} and denoted by $\TopUpper$. The topology of uniform biconvergence on bounded sets with respect to the lower uniform structure of $X$ will be called the \emph{lower topology} and denoted by $\TopLower$.
\end{definition}

These two topologies are in general different, as shown by the following two examples.

\begin{example}\label{ex:symupper}
Let $X$ be an infinite countable set. Let $G_1 = \Sym(X)$, endowed with its usual Polish topology of pointwise convergence. All automorphisms of $G_1$ (as an abstract group) are inner, hence a fortiori bicontinuous. Moreover, $G_1$ is centrefree, hence $G_2 = \Aut(G_1)$ is abstractly isomorphic to $G_1$. Since $G_1$ is Roelcke-precompact \cite[9.14]{RD_1981}, it is a fortiori locally bounded. The upper topology on $G_2$ is thus simply the topology of upper uniform convergence. We claim that this topology is nothing but the discrete topology on $\Sym(X)$. Therefore, $G_2$ is again locally bounded, and $G_3 = \Aut(G_2)$, still abstractly isomorphic to $G_1$, is now endowed with the topology of pointwise convergence \emph{on $G_2$}. We claim that this topology is also discrete, hence $G_3$ is actually isomorphic to $G_2$. We observe moreover that the conjucation morphism $\conj\colon G_1 \rightarrow G_2$ (given by $\conj(x)(y) = xyx\inv$) is not continuous for the upper topology since $G_2$ is discrete (compare with Section~\ref{sec:conjcont} below).

Let us prove our two claims about the topologies on $\Sym(X)$, identifying $X$ with $\Ent$ for simplicity. Let $\conj(x_i)$ be a net converging to the identity in $G_2$. In particular, there is an index $j$ such that for any $i \geq j$ and any $\sigma \in G_1$, the permutation $\left(\conj(x_i)(\sigma)\right)\inv \sigma$ fixes the point $0 \in \Ent$ (since $G_1$ is bounded and the fixator of a point is an identity neighborhood in $G_1$). Choosing for $\sigma$ first the inversion $\tau\colon x \mapsto -x$ and then the shifts $\rho^n\colon x \mapsto x + n$, we compute that $x_i$ is the identity for $i \geq j$. Hence the topology on $G_2$ is discrete.

Let now $\conj(x_i)$ be a net converging to the identity in $G_3$. In particular, for any finite set $F \subseteq G_2$, there is an index $j$ such that for any $i \geq j$ and $\sigma \in F$, the permutation $\left(\conj(x_i)(\sigma)\right)\inv\sigma$ is trivial (since $G_2$ is discrete), that is, $x_i\inv$ commutes with $\sigma$. Choosing for $F$ a set with a trivial centraliser (for instance, $F =\{\tau, \rho\}$), we have $x_i = \id_X$. Hence the topology on $G_3$ is also discrete.
\end{example}

\begin{example}\label{ex:symlower}
Now consider again the Polish group $G_1 = \Sym(\Ent)$ as in the previous example but endow $G_2 = \Aut(G_1)$ with the lower topology, that is, the topology of Roelcke-uniform convergence since $G_1$ is bounded.  We claim that the lower topology is nothing but the usual Polish topology $\Top$ of pointwise convergence on $\Ent$ when we identify $G_2$ to $G_1$ via the (inverse of) the conjugation morphism $\conj$. Hence in particular the lower topology is strictly weaker than the upper topology.

Indeed, basic identity neighborhoods for the lower topology $\TopLower$ are given by
\begin{equation*}
\mathbf{W}_V = \{ \phi \in \Aut(X)\ |\ \forall \sigma \in X,\ \phi(\sigma) \in V \sigma V\}
\end{equation*}
for symmetric identity neighborhoods $V$ of $X$. Obviously, if $\phi = \conj(x)$ and $x \in V$, then $\phi \in \mathbf{W}_V$. Hence $\Top$ is finer than $\TopLower$. (The continuity of $\conj$ holds more generally, see Proposition~\ref{prop:conjcontlower} below.)

Reciprocally, let $\phi \in \mathbf{W}_V$ and $x \in X$ be such that $\phi = \conj(x)$. Without loss of generality, we can assume that $V$ is the fixator of $n$ points $k_1, \dots, k_n \in \Ent$. Choose then $n$ permutations $\sigma_i$ such that $k_i$ is the \emph{only} fixed point of $\sigma_i$. Since $x \sigma_i x\inv \in V \sigma_i V$, the permutation $x$ must also fix $k_i$, for any $i = 1, \dots, n$. Hence $x \in V$. Thus $\TopLower$ is finer than $\Top$ and both topologies agree. In this example, the conjugation morphism $\conj\colon G_1 \rightarrow \Aut(G_1)$ is thus an isomorphism of topological groups.

(A deeper reason for the equality of the lower topology with the original one will be given in Remark~\ref{rem:loweroriginal} below.)
\end{example}

However, the lower and upper topologies agree for a large class of groups. Let us say that a topological group is \emph{coarsely SIN} if for any identity neighborhood $U$ and any bounded set $B$, there is an identity neighborhood $V$ such that $xVx\inv \subseteq U$ for all $x \in B$. Examples include of course SIN groups (that is, groups admitting a basis of conjugation-invariant identity neighborhoods) as well as all locally compact groups \cite[II.4.9]{HR_1979}. On the other hand, a bounded non-SIN group, such as $\Sym(\Nat)$ or $\Homeo_+ [0,1]$, cannot be coarsely SIN. The point of this notion is the following easy fact:

\begin{proposition}\label{prop:csinlowup}
Let $X$ be a locally bounded coarsely SIN group. Then the lower and upper topologies agree on $\Aut(X)$.
\end{proposition}

\begin{proof}
 As $\TopUpper$ is finer than $\TopLower$, we only need to show that any identity neighborhood $W$ of the former contains an identity neighborhood $W'$ of the latter. We can assume without loss of generality that
\begin{equation*}
W = \ens{g \in \Aut(X)}{\forall x \in B,\ g(x), g\inv(x) \in Ux \cap xU}
\end{equation*}
for some bounded set $B$ and identity neighborhood $U$ in $X$, and moreover that $B$ is symmetric and contains the identity. Let $U'$ be an identity neighborhood such that $U'U' \subseteq U$. Since $X$ is coarsely SIN, we can find some identity neighborhood $V$ such that $xVx\inv \subseteq U'$ for all $x \in B$. Since $B$ is symmetric and contains the identity, we have $VxV = V(xVx\inv)x \subseteq Ux$ and $VxV = x(x\inv Vx)V \subseteq xU$ for any $x \in B$. Consequently, the set
\begin{equation*}
W' = \ens{g \in \Aut(X)}{\forall x \in B,\ g(x), g\inv(x) \in VxV},
\end{equation*}
which is an identity neighborhood for $\TopLower$, is contained in $W$.
\end{proof}

\begin{remark}\label{rem:topbraconn}
	In particular, the lower and upper topologies agree for locally compact groups, for which this common topology is known as the \emph{Braconnier topology} (a survey of which can be found as Appendix~I of \cite{CM_2011}). Alternatively to Proposition~\ref{prop:csinlowup}, we could also rely on a more general topological fact, namely: the topology of uniform convergence on compact sets on a uniform space $X$ actually depends only on the topology of $X$ and not on its uniform structure \cite[\textsc{x}, \S~3, n°~4]{Bourbaki_TG_V}. 
\end{remark}

\begin{remark}
Just as any subgroup of a SIN group is still SIN \cite[3.24]{RD_1981}, any subgroup $H$ of a coarsely SIN group $G$ is still coarsely SIN. Indeed, a bounded set in $H$ stays bounded in $G$.
\end{remark}

\subsection{Conjugation morphism}\label{sec:conjcont}

Recall that there is a canonical morphism $\conj$ from a topological group $X$ to its automorphism group $\Aut(X)$ given by conjugation, namely $\conj(x)(y) = xyx\inv$. When $\Aut(X)$ is endowed with the upper topology, the conjugation morphism may fail to be continuous if the group is not coarsely SIN---as we already know from Example~\ref{ex:symupper}. The situation is nicer for the lower topology.

\begin{proposition}\label{prop:conjcontlower}
	Let $X$ be a locally bounded group. Then the conjugation morphism $\conj\colon X \rightarrow \Aut(X)$ is continuous if $\Aut(X)$ is endowed with the lower topology.
\end{proposition}

\begin{proof}
	Since $\conj$ is a group morphism, we only need to consider continuity at the identity with respect to the topology of (lower) uniform convergence on bounded sets. But a basic neighborhood of $\id_X$ for the latter is given by
	\begin{equation*}
		\ens{g \in \Aut(X)}{\forall x \in B, g(x) \in VxV},
	\end{equation*}
	with $B$ a bounded set and $V$ an identity neighborhood of $X$. This set obviously contains $\conj(V \cap V\inv)$, hence $\conj$ is continuous.
\end{proof}

\begin{remark}\label{rem:conjcontroelckecomp}
Actually, we have a bit more. Assume that $X$ is contained as a normal subgroup of a topological group $Y$ (not necessarily closed). Then the action of $Y$ by conjugation on $X$ gives a morphism $\alpha\colon Y \rightarrow X$. The above proof shows that $\alpha$ is continuous when $\Aut(X)$ is endowed with the topology $\Top$ of uniform biconvergence on bounded sets \emph{with respect to the uniform structure $\Unif$ induced on $X$ by the lower uniform structure of $Y$}. In general, $\Unif$ is coarser than $\UnifRoelcke$, the lower uniform structure of $X$ \cite[3.25]{RD_1981} and therefore the topology $\Top$ is not a priori compatible with the group structure of $\Aut(X)$. However, $\Unif = \UnifRoelcke$ if $X$ is open or dense in $Y$ \cite[3.24]{RD_1981}. Moreover, $\Top = \TopLower$ if $X$ is locally compact by Remark~\ref{rem:topbraconn} (compare with \cite[26.7]{HR_1979}).
\end{remark}

\subsection{Completeness}\label{sec:complgroup}

We know study the completeness of $\Aut(G)$ with respect to the lower and upper topologies. As an immediate corollary of Theorem~\ref{th:raikov} and Proposition~\ref{prop:lruagree}, we have:

\begin{corollary}\label{cor:raikovupper}
	Let $X$ be a locally bounded Ra\u{\i}kov-complete group. Then $\Aut(X)$ is also Ra\u{\i}kov-complete for the upper topology.
\end{corollary}

The assumption on $X$ is rather weak: it holds in particular for locally compact groups and for Polish groups. However, Example~\ref{ex:symupper} has shown that, for non-coarsely SIN groups, the upper topology can be quite strong and it would be desirable to yield completeness for the lower one. Of course, if $X$ happens to be Roelcke-complete, then $\Aut(X)$ would be Ra\u{\i}kov-complete for the lower topology thanks to Theorem~\ref{th:raikov}. Unfortunately, Roelcke-completeness can be quite a strong assumption, since it implies in particular that the group is Weil-complete (complete for the left or right uniform structures), hence would fail for many Polish groups, namely those without a complete left-invariant metric\footnote{Actually, for Polish groups, Roelcke-completeness is equivalent to Weil-completeness \cite[11.4]{RD_1981}. We note in passing that, once again, these refinements are void for locally compact groups, which are automatically Roelcke-complete, as is easily shown \cite[8.8]{RD_1981}.}. Hence the interest of the following theorem, whose proof has a flavour similar to that of Theorem~\ref{th:raikov}.

\begin{theorem}\label{th:raikovlower}
Let $X$ be a locally bounded Ra\u{\i}kov-complete group. Then $\Aut(X)$ is also Ra\u{\i}kov-complete for the lower topology.
\end{theorem}

\begin{proof}
	Let us write $\Unif$ for the upper uniform structure of the lower topology and let $(g_\alpha)$ be a Cauchy net in $\Aut(X)$ for $\Unif$. As we have shown in the proof of Theorem~\ref{th:raikov}, $(g_\alpha)$ is also a Cauchy net for the structure $\UnifConv$, the uniform structure of biconvergence on bounded sets on $X$ (with respect to the Roelcke uniform structure on $X$). Since these two uniform structures induce the same topology on $\Aut(X)$, we only need to prove that the net $(g_\alpha)$ admits a pointwise limit, thanks again to~\bbkitgx{1}{5}{prop.~5}.

Here the key-point is the eventual equicontinuity of Cauchy nets, more precisely:
\begin{lemma}
	If $(g_\alpha)$ is a Cauchy net in $\Aut(X)$ for $\UnifConv$, then for any identity neighborhood $U$, there are an index $\alpha_0$ and an identity neighborhood $V$ such that $g_\alpha (V) \subseteq U$ for all $\alpha \geq \alpha_0$.
\end{lemma}

\begin{proof}[Proof of the lemma]
	Let $W$ be a bounded identity neighborhood in $X$ such that $W^3 \subseteq U$. Since $(g_\alpha)$ is Cauchy for $\UnifConv$, there is in particular some index $\alpha_0$ such that for all $\alpha, \beta \geq \alpha_0$, we have $(g_\alpha, g_\beta) \in \Entsc_{W, W^\Roelcke}$. That is, for all $x \in W$ :
	\begin{equation*}
		g_\alpha (x) \in W g_\beta (x) W.
	\end{equation*}
	Now since $g_{\alpha_0}$ is continuous, we can find an identity neighborhood $V \subseteq W$ such that $g_{\alpha_0} (V) \subseteq W$. Therefore, for any $x \in V$ and $\alpha \geq \alpha_0$, we have $g_\alpha (x) \in W g_{\alpha_0} (x) W \subseteq W^3$. That is, $g_\alpha (V) \subseteq U$, as required.
\end{proof}

We can now conclude the proof of Theorem~\ref{th:raikovlower}. Since $X$ is Ra\u{\i}kov-complete, we only need to show that, for any $x \in X$, the net $(g_\alpha (x))$ is Cauchy for the upper uniform structure of $X$. So let $U$ be an identity neighborhood and let $\alpha_0$ and $V$ be as in the lemma. Let moreover $W$ be an identity neighborhood (depending on $x$) such that $WxW \subseteq Vx \cap xV$. Since $(g_\alpha)$ is Cauchy for $\Unif$, there is an index $\alpha_1 \geq \alpha_0$ such that $(g_\beta\inv \circ g_\alpha) (x) \in WxW$ for any $\alpha, \beta \geq \alpha_1$ (choosing $B$ to be the singleton $\{x\}$).

Therefore, we have $(g_\beta\inv \circ g_\alpha) (x) \in Vx \cap xV$, hence, by our choice of $V$, $g_\alpha (x) \in Ug_\beta(x) \cap g_\beta(x) U$. Since $U$ was arbitrary, this means that $(g_\alpha (x))$ is Cauchy for the upper uniform structure, hence is convergent. Thus $(g_\alpha)$ admits a pointwise limit, which concludes the proof of Theorem~\ref{th:raikovlower}.
\end{proof}

This has interesting corollaries for Polish groups.

\begin{corollary}\label{cor:polaut}
	Let $X$ be a Polish locally bounded group. Then $\Aut(X)$ is metrisable and Ra\u{\i}kov-complete for any of the lower or the upper topologies. In particular, if $\Aut(X)$ is separable for the lower topology, it is Polish.
\end{corollary}

\begin{proof}
	We already know that $\Aut(X)$ is Ra\u{\i}kov-complete for these two topologies by Corollary~\ref{cor:raikovupper} and Theorem~\ref{th:raikovlower}, since Polish groups are so. For the uniform structure of biconvergence on bounded sets to be metrisable, it needs to have a countable basis of entourages, which is the case as soon as both the uniform structure and the bornology on $X$ admit a countable basis (Proposition~\ref{prop:autulbmetrisable}). The first condition means that the uniform structure is metrisable \bbkitgix{2}{4}{th.~1}, which is the case for any of the four natural uniform structures on a metrisable group. The second condition is equivalent to local boundedness for Polish groups \cite[2.28]{Rosendal_coarse}. Hence both conditions are met for Polish locally bounded groups.
\end{proof}

\begin{remark}
Corollary~\ref{cor:polaut} also implies that, for Polish locally bounded groups, a necessary condition for the \emph{upper} topology to be separable is that it coincides with the lower one. Indeed, if the upper topology is separable, then so is the lower one, hence both topologies are Polish. Since they are comparable, they must coincide (see e.g.~\cite[1.2.6]{BK_1996}). This necessary condition is however not sufficient (Example~\ref{ex:autnonsep} below).
\end{remark}

\begin{corollary}\label{cor:polautlocroelcke}
	Let $X$ be a Polish locally Roelcke-precompact group. Then $\Aut(X)$ is a Polish group for the lower topology.
\end{corollary}

\begin{proof}
	Since Roelcke-precompact subsets are bounded, we can apply Corollary~\ref{cor:polaut} and we only need to show that $\Aut(X)$ is separable. A continuous automorphism of $X$ is automatically Roelcke-uniformly continuous, hence any $g \in \Aut(X)$ extends uniquely to a homeomorphism of $X^\Roelcke$, the Roelcke-completion of $X$. By construction, $X^\Roelcke$ is a separable metrisable uniform space inducing on $X$ the Roelcke uniform structure. Moreover, Zielinski has shown that $X^\Roelcke$ is locally compact \cite[Th.~16]{Zielinski_2019} and that bounded sets are exactly the Roelcke-precompact ones in a Polish locally Roelcke-precompact group \cite[Th.~14]{Zielinski_2019}. Therefore, we can view $\Aut(X)$ with its lower topology as a subgroup of the homeomorphism group of a second-countable metrisable locally compact space, endowed with the usual topology of uniform biconvergence on compact sets. The latter is second-countable \bbkitgx{3}{3}{cor.}, hence so is $\Aut(X)$. A fortiori, $\Aut(X)$ is separable.
\end{proof}
 
 Once again, Example~\ref{ex:symupper} has shown that $\Aut(X)$ can fail to be separable in the upper topology, even when $X$ is a Polish Roelcke-precompact group. Moreover, the next example shows that we cannot either weaken the \enquote{locally Roelcke-precompact} assumption of the above Corollary by a mere \enquote{locally bounded}.

\begin{example}\label{ex:autnonsep}
Let $V = \ell^2 (\Ent)$ be a real Hilbert space and $G$ be its underlying additive group. In particular, $G$ is a Polish locally bounded SIN group. Any \emph{continuous} group morphism of $G$ is actually a linear map of $V$, therefore $\Aut(G)$ is nothing but the group $\GL(V)$ of linear automorphisms of $V$. Moreover, the lower topology on $\Aut(G)$ is the usual norm-operator topology on $\GL(V)$. Indeed, the latter is, by definition, the topology of uniform convergence on bounded sets on $V$, which happens to agree with the topology of biconvergence\footnote{This follows easily from the equality
\begin{equation*}
\lVert T\inv \rVert = \inf \ens{\lVert Tx \rVert_V}{\lVert x \rVert_V = 1}.
\end{equation*}}.

But $\GL(V)$ is not separable for the norm-operator topology. Indeed, for any permutation $\sigma \in \Sym(\Ent)$, the map $T_\sigma$ defined coordinatewise by $T_\sigma (\delta_n) = \delta_{\sigma\inv (n)}$ is a linear automorphism of $V$. However, for any two distinct permutations $\sigma, \tau$, the norm of the operator $T_\sigma - T_\tau$ is $2$. Since $\GL(V)$ thus contains an uncountable set of points uniformly apart from each other, it cannot be separable.
\end{example}

Moreover and unfortunately, as the next example shows, $\Aut(X)$ need not be locally bounded for the lower topology, even if $X$ is locally Roelcke-precompact. Thus Corollary~\ref{cor:polautlocroelcke} cannot be iterated.

\begin{example}
Let $\{p_i\}$ be a collection of infinitely many distinct primes. Let, for any $i$, $H_i$ be a finitely generated $p_i$-group whose automorphism group is infinite. (For instance, $H_i$ can be the free Burnside group of rank $2$ and exponent $p_i \geq 665$. Its automorphism group is infinite since its center cannot be of finite index.)

Let $G$ be the direct sum $\bigoplus_i H_i$, endowed with the discrete topology. It is countable, hence in particular is a Polish locally compact group. Observe that an element $(h_i)_i$ in $G$ has order a power of $p_j$ if and only if $h_i$ is the neutral element of $H_i$ for all $i \neq j$. Therefore, any automorphism of $G$ has to preserve each factor $H_i$, hence the group $\Aut(G)$ can be identified to $\prod_i \Aut(H_i)$.

Since $G$ is discrete, the lower topology on $\Aut(G)$ is nothing but the topology of pointwise convergence. Hence $\Aut(G)$ is also topologically isomorphic to $\prod_i \Aut(H_i)$, the latter being endowed with the product topology of the \emph{discrete} groups $\Aut(H_i)$ (recall that $H_i$ is finitely generated, hence the topology of pointwise convergence on $\Aut(H_i)$ is discrete).

However, $\prod_i \Aut(H_i)$, being an infinite product of infinite discrete groups, is not locally bounded \cite[3.30]{Rosendal_coarse}.
\end{example}

\begin{remark}\label{rem:loweroriginal}
Corollary~\ref{cor:polautlocroelcke} sheds a new light on the computation of Example~\ref{ex:symlower}. Indeed, if $X$ is any \emph{centrefree} Polish locally Roelcke-precompact group \emph{without outer automorphism}, then the lower topology on $\Aut(X) \simeq X$ is a Polish topology by Corollary~\ref{cor:polautlocroelcke} which is weaker than the original one by Proposition~\ref{prop:conjcontlower}. They must thus agree by \cite[1.2.6]{BK_1996}.\end{remark}

\section{Further examples and counterexamples}\label{sec:examples}

We give here various examples of topological groups that can be encompassed in the framework of uniformly locally bounded spaces. Of course, the reader interested in some specific example can easily check by hand that the topology is compatible with the group structure. The interest of the more general framework is twofold. First, as the variety of the examples should show, it often \emph{suggests} new topologies to consider on well-known groups. Second, from a technical point of view, some properties (such as Ra\u{\i}kov-completeness) come now for free.

In this section, in order to yield completeness results, we will repeatedly and without much notice use the following \enquote{fact}, already explicitly used in the case of the automorphism group of a topological group $\Aut(G)$. If we can view the automorphism group $\Aut(X)$ of some topologico-algebraic structure $X$ as a subgroup of $\Autulb(X)$, then it is a closed subgroup, since preserving some algebraic structure is a condition closed for the pointwise convergence topology and all elements of $\Autulb(X)$ are continuous.

\subsection{Metric spaces}\label{sec:metricspaces}

Let $X$ be a metric space and $G$ be its isometry group. By Theorem~\ref{th:comp}, $G$ is a topological group for the topology of uniform biconvergence on bounded sets (where \enquote{bounded} has here the usual meaning of \enquote{with finite diameter}). On the other hand, since $G$ is equicontinuous, it is also a topological group for the topology of pointwise convergence \bbkitgx{3}{5}{cor.}. If $X$ is \emph{proper}, that is, if bounded sets in $X$ are relatively compact, then both topologies agree on $G$.

In general, these topologies are distinct. For instance, if $X$ is a real Hilbert space, then $G$ is the semidirect product $\Orth (X) \ltimes X$. On the subgroup $\Orth(X)$ of linear isometries, the topology of uniform biconvergence on bounded sets is the norm-operator topology, whereas the topology of pointwise convergence is the strong operator topology. These two topologies are easily shown to disagree whenever $X$ is infinite-dimensional (for instance, with the notations of Example~\ref{ex:autnonsep}, the sequence $(T_{\sigma_i})$ converges in the strong operator topology to the identity whenever the sequence of permutations $(\sigma_i)$ converges pointwise to the identity; but it does not converge in the norm-operator topology).

Harvesting results from Sections~\ref{sec:topprop} and~\ref{sec:completeness}, we see that $G$ is always metrisable and is Ra\u{\i}kov-complete if $X$ is complete, since $G$ is closed in $\Autulb(X)$ (even for the pointwise convergence topology). It is in general not separable, as shown by the above example of Hilbert spaces. If $X$ is ultrametric, then $G$ is non-Archimedean. Lastly, for any $x \in X$, the fixator $G_x$ is SIN for this topology, since the balls centered at $x$ form a basis of the bornology.

\subsection{Locally compact spaces}\label{sec:loccomp}

	Although all locally compact spaces are completely regular (and therefore admit a compatible uniform structure), not every locally compact space can be viewed as a \emph{uniformly} locally compact space (that is, a uniformly locally bounded space with relatively compact subsets as bounded subsets). Indeed, since compact spaces are complete, a uniformly locally compact space has to be complete by Remark~\ref{rem:intertwining}.(\ref{rem:intertwining:complete}). But there exist locally compact spaces without any complete compatible uniform structures, such as $\omega_1$ (the space of all countable ordinals, endowed with the order topology). See also \cite[\S~6]{Dieudonne_1939_complet}. (Other examples are provided by the following necessary condition: a uniformly locally compact \emph{connected} space has to be $\sigma$-compact \cite[th.~\textsc{ix}]{Weil_uniforme}, see also Remark~\ref{rem:intertwining}.(\ref{rem:intertwining:sigma}).)

Nonetheless, the results of this note can be applied to any locally compact space $X$, thanks to the well-known detour via Alexandrov's one-point compactification~$\Alex{X}$. Recall that the latter space is $X \cup \{\infty\}$, where the first component is a homeomorphic image of $X$ and a basis of neighborhoods of $\infty$ is given by complements in $\Alex{X}$ of compact sets in $X$. The space $\Alex{X}$ is compact \cite[\textsc{i}, \S~9, n°~8]{Bourbaki_TG_I} and as such uniformly locally bounded (for its unique uniform structure and the compact bornology, which is trivial). Any homeomorphism of $X$ extends uniquely as a homeomorphism of $\Alex{X}$ fixing the point $\infty$, and reciprocally any homeomorphism of $\Alex{X}$ fixing $\infty$ induces a homeomorphism of $X$. Therefore, the abstract group $\Homeo(X)$ can be identified to the fixator $F$ of $\infty$ in $\Homeo(\Alex{X})$. Moreover, it is routinely checked that the topology of uniform biconvergence on compact sets on the former is equivalent to the topology of uniform convergence on $F$ \bbkitgx{3}{5}{prop.~12}. Since $F$ is closed in $\Homeo(\Alex{X})$ (for the pointwise topology, hence a fortiori for the topology of uniform convergence), Theorems~\ref{th:comp} and \ref{th:raikov} and Proposition~\ref{prop:autulbmetrisable} yield:

\begin{corollary}[Arens]
	Let $X$ be a locally compact space. Endowed with the topology of uniform biconvergence on compact sets, $\Homeo(X)$ is a Ra\u{\i}kov-complete topological group. If $X$ is metrisable and $\sigma$-compact, then $\Homeo(X)$ is metrisable.
\end{corollary}
\subsection{Other bornologies in groups}\label{sec:otherborn}

We think that, on the automorphism group $\Aut(G)$ of a topological group $G$, the lower topology, namely the combination of the lower uniform structure and the bounded sets as defined by Rosendal, achieves a balance between the variety of examples to which it applies and yielding non-trivial topologies. However, it can be checked that many results of Section~\ref{sec:groups} would still hold for other uniform structures or other bornologies.

As for bornologies, the reader can indeed check that Section~\ref{sec:groups} only relies on the fact that the standard bornology is \enquote{compatible with the topological group structure}, that is: if $A$ and $B$ are bounded and if $\phi$ is a continuous group endomorphism, then $\phi(AB\inv)$ is bounded. As an extreme case, this holds of course for the trivial bornology (thus for the topology of lower uniform convergence). Moreover, we can also only require the bornology to be compatible with the group structure (if $A$ and $B$ are bounded, then so is $AB\inv$), up to considering only the subgroup of $\Aut(G)$ that preserves this bornology. These compatible bornologies have been studied e.g. in \cite{NR_2012} (under the name \enquote{generating family}).

\begin{example}
The compact bornology is compatible with the topological group structure. Under countability assumptions, this does not bring something new, but in a locally compact group which is \emph{not} $\sigma$-compact, we could \emph{a priori} find bounded sets that are not relatively compact. The lower topology could then be strictly finer than the topology of uniform biconvergence on compact sets (i.e., the Braconnier topology); but our results apply to both topologies by the above discussion. A similar comment can be made for the bornology of Roelcke-precompact subsets: it is compatible with the group structure if the group is locally Roelcke-precompact \cite[Prop.~12]{Zielinski_2019} but it could be smaller than the standard bornology if the group is not Polish.
\end{example}

On the other hand, we did rely more on the specific form of the four natural uniform structures, that interact nicely with the group structure. We could have added to this landscape the fine uniform structure, but Example~\ref{ex:symupper} should warn us against using too fine uniform structures, that rapidly produce discrete topologies on $\Aut(G)$. On the positive side, the fine uniform structure is very often complete---it is for instance always the case for the fine uniform structure on a metric space \cite[\S~4]{Dieudonne_1939_complet}.

\subsection{Exotic topologies}\label{sec:exotic}

Theorem~\ref{th:comp} can be used to yield unusual topologies on some groups. Let $G$ act on a uniform space $X$, say by uniform isomorphisms for simplicity. To appeal to Theorem~\ref{th:comp}, we need to find a bornology on $X$ that is $G$-invariant and makes $X$ uniformly locally bounded---so that $G$ can be viewed as a subgroup of $\Autulb(X)$.

Let us illustrate this strategy on a simple example: $X$ is uniformly discrete. In particular, it is uniformly locally bounded with respect to any bornology and any bijection of $X$ is uniformly continuous. On $\Sym(X)$, this does not yield new topologies: it is easy to check that the only $\Sym(X)$-invariant bornologies on $X$ are $\Born_\kappa$, the subsets of $X$ of cardinality $< \kappa$, where $\kappa$ is an infinite cardinal. (In particular, if $X$ is countable, this only gives the pointwise convergence topology and the discrete topology.)

However, some subgroups of $\Sym(X)$ may preserve other, more interesting bornologies. For instance, the automorphism group $G = \Aut(\Rat, <)$ of the set of rationals with its usual order preserves the following bornologies:
\begin{align*}
\Born &= \ens{B}{\exists a, b \in \Rat:\ B \subseteq [a, b]}, \\
\Born_+ &= \ens{B}{\exists a \in \Rat:\ B \subseteq [a, \infty[}, \\
\Born_- &= \ens{B}{\exists b \in \Rat:\ B \subseteq [-\infty, b]}.
\end{align*}
The corresponding three topologies $\Top$, $\Top_+$ and $\Top_-$ are easily seen to be non-discrete, distinct from each other and from the usual poinwise convergence topology. Moreover, $\Top$ is weaker than $\Top_+$ and $\Top_-$, and the supremum of the latter two is the discrete topology. All these topologies are non-Archimedean (Remark~\ref{rem:nonarchgroup}).

Although they are all metrisable and Ra\u{\i}kov-complete ($\Aut(\Rat, <)$ is closed in $\Sym(\Rat)$ for the pointwise convergence topology, hence a fortiori for these finer topologies), they are not Polish (that is, separable) because distinct Polish topologies on the same group have to be incomparable (see for instance \cite[1.2.6]{BK_1996}).

Of course, similar bornologies can be defined on any lattice $(X, <)$, and yield non-discrete topologies on the group $\Aut(X, <)$ if for any interval $[a, b]$, $[a, \infty[$ or $]-\infty, b]$, there is a non-trivial automorphism fixing pointwise this interval. (In particular, $(X, <)$ should not admit a minimum or a maximum.)

Other exotic examples of this flavour can be provided by $\Aut(\Graph)$, where $\Graph$ is a graph. Any numeric graph invariant can help to define an $\Aut(\Graph)$-invariant bornology. For instance, we could say that a subset of the vertex set of $\Graph$ is bounded if it induces a subgraph with finite chromatic number, or finite maximal degree, or finite diameter, or finite Cheeger constant, and so on.

\subsection{Automorphism group of \texorpdfstring{$\sigma$}{sigma}-algebras}

Let $\Alg$ be a $\sigma$-algebra on a set $X$, and assume for simplicity that $\Alg$ contains all singletons. (For instance, $\Alg$ can be the Borel $\sigma$-algebra of a Hausdorff topological space.) The automorphism group $G = \Aut(X, \Alg)$ (that is, the set of all bimeasurable bijections of $X$) can be endowed with various topologies, as instances of Theorem~\ref{th:comp}, as follows.

The uniformly locally bounded space on which $G$ acts will not be seen on $X$ directly but on another space. To this end, remark that $G$ can be \emph{identified} to the group $\Aut(\Alg, \subseteq)$ of automorphisms of the \emph{set} $\Alg$ ordered by inclusion. Indeed, such an automorphism $\phi$ preserves the atoms\footnote{An \emph{atom} of a Boolean algebra is an immediate successor of the minimum element.} of $\Alg$, which are the singletons by our assumption, and preserving the order $\subseteq$ then ensures that the action of $\phi$ on $\Alg$ is induced by its action on the singletons. This immediately suggests a group topology to consider on $\Aut(X, \Alg)$: the pointwise convergence topology of $\Aut(\Alg, \subseteq)$.

Other topologies arise as follows. Let $\mu$ be any measure on $\Alg$. We can use it to define a (non-Hausdorff) uniform structure $\Unif_\mu$ on $\Alg$ via the basis of uniform entourages
\begin{equation*}
\ens{(A, B) \in \Alg^2}{\mu(A \symdiff B) < \epsilon}.
\end{equation*}
Observe that $\Unif_\mu$ is complete (a Cauchy sequence $(A_n)$ converges to $\bigcap_n \bigcup_{k \geq n} A_k$).

If $\mu_1, \dots, \mu_k$ are measures on $\Alg$, the supremum of the uniform structures $\Unif_{\mu_1}, \dots, \Unif_{\mu_k}$ is the structure $\Unif_\mu$, where $\mu = \mu_1 + \dots + \mu_k$. Therefore, for \emph{any} family $\Fam$ of measures on $\Alg$, the supremum $\Unif_\Fam$ of all the structures $\Unif_\mu$ for $\mu \in \Fam$ is a \emph{projective} limit of complete uniform spaces, and is thus itself complete \bbkitgii{3}{1}{prop.~4}. Moreover, $\Unif_\Fam$ is Hausdorff if and only if $\Fam$ \emph{separates the elements} of $\Alg$ (that is, for any $A \neq B \in \Alg$, there exists $\mu \in \Fam$ with $\mu(A \symdiff B) \neq 0$), which is the case for instance if $\Fam$ contains all Dirac masses.

If non-$\mu$-null sets have measure bounded away from zero (that is, if $\mu(\Alg) \subseteq \{0\} \cup [\epsilon, \infty[$ for some $\epsilon > 0$)---for instance if $\mu$ only takes values $0$ and $1$---, then $\Unif_\mu$ is non-Archimedean. Moreover, a supremum of non-Archimedean uniform structures is also non-Archimedean.

Combining these observations with the results of Section~\ref{sec:topcomp} (with the trivial bornology), and keeping in mind that $\Aut(\Alg, \subseteq)$ is closed in $\Autulb(\Alg, \Unif_\Fam)$, we get:
\begin{proposition}
The group $G = \Aut(X, \Alg)$, endowed with the topology $\Top_\Fam$ of uniform convergence with respect to $\Unif_\Fam$, is a Ra\u{\i}kov-complete topological group. Moreover:
\begin{itemize}
\item If $\Fam$ separates elements of $\Alg$, then $G$ is Hausdorff.
\item If $\Fam$ is moreover countable, then $G$ is metrisable.
\item If each measure of $\Fam$ have non-zero values bounded away from zero, then $G$ is non-Archimedean.
\item If a subgroup $H$ of $G$ preserves each measure of $\Fam$, then $H$ is SIN.
\end{itemize}
\end{proposition}

Observe that the Hausdorffisation of $G$ with respect to $\Top_\Fam$ is nothing but the equivalence relation \enquote{$g \sim h$ if for each $A \in \Alg$, the difference $g(A) \symdiff h(A)$ has null measure for every $\mu \in \Fam$}. Remark also that, if $\mu$ is the trivial measure ($\mu(A) = \infty$ for any nonempty $A \in \Alg$), the structure $\Unif_\mu$ is discrete. Therefore, these topologies are more interesting if we restrict our attention to $\sigma$-finite measures or, even better, to probability measures.

A first interesting example is given by the family $\Prob$ of \emph{all} possible probability measures on $\Alg$. We could also consider the subfamily $\Prob_c$ of all atomless ones, or on the contrary the subfamily $\Prob_\beta$ of all probability measures that take only values $0$ and $1$. In the particular case where $\Alg$ is the Borel $\sigma$-algebra of a standard Borel space, the first two topologies have been introduced and studied in \cite{BDK_2006} (under the name $\tau$ and $\tau_0$; their third group topology $p$ is the pointwise convergence topology on $\Aut(\Alg, \subseteq)$).

\begin{remark}
Alternatively, we can also \emph{identify} $G$ to the group $\Isom_+(\Measb(\Alg))$ of all order-preserving linear isometries of $\Measb(\Alg)$, where $\Measb(\Alg)$ is the Banach sublattice of $\ell^\infty (X)$ made of $\Alg$-measurable maps. Indeed, the indicators of atoms (that is, in our cases, the Dirac masses $\delta_x$) can be defined in $\Measb(\Alg)$ as those positive norm-one maps $f$ such that for any $g \in \Measb(\Alg)$, if $0 \leq g \leq f$, then $g$ is a multiple of $f$. Since any elements of $\Measb(\Alg)$ is the difference of two positive ones and since, moreover, for any positive $f \in \Measb(\Alg)$ and $x \in X$, the value $f(x)$ is the supremum of all $\alpha$ such that $f \geq \alpha \delta_x$, we see that any order-preserving linear isometry of $\Measb(\Alg)$ is determined by its action on the Dirac masses, hence is induced by a unique element of $G$.

	We can now view $\Measb(\Alg)$ as a uniformly locally bounded spaces in various ways. The first one comes from its structure of a normed space, hence yields the common operator-norm topology on $\GL(\Measb(Alg))$. However, this topology induces the discrete topology on $G$ (because two distinct Dirac masses are at distance $2$). In a completely analoguous way as for $\Alg$, we can also endow $\Measb(\Alg)$ with the uniform structure generated by $\ens{(f, g)}{\abs{\mu(f) - \mu(g)} < \epsilon}$ for some measure $\mu$, and consider the supremum of these structure for various measures. In the standard Borel case, \cite{BDK_2006} showed that this amounts to the same topology as the one yielded by the above discussion.

Note that, $\Isom(\Measb(\Alg))$ being equicontinuous, we can also consider the topology of pointwise convergence, namely the strong operator topology. This is actually the same as the pointwise convergence topology of $\Aut(\Alg, \subseteq)$ \cite{BDK_2006}
\end{remark}

\subsection{Galois groups}\label{sec:galois}

Let $K$ be a field and let $L$ be a Galois extension of $K$. Endowing $L$ with the discrete uniform structure, we can view the Galois group $\Gal(\extens{L}{K})$ as a closed subgroup of $\Autulb(L)$ for various bornologies on $L$:
\begin{itemize}
\item $\Bornfin$: $B$ is \emph{bounded} if it is finite;
\item $\BornKrull$: $B$ is \emph{bounded} if it is contained in some finite subextension of $\extens{L}{K}$;
\item $\Borndegree$: $B$ is \emph{bounded} if there is a bound on the degree of elements in $B$;
\end{itemize}

Since, by definition, elements of $\Gal(\extens{L}{K})$ fix each point of $K$, the first two bornologies define the same topology on $\Gal(\extens{L}{K})$, which is nothing but Krull's usual profinite topology $\TopKrull$ \cite[\textsc{v}, \S~{10}, n°~{3}]{Bourbaki_A_IV}. Let $\Topdegree$ be the other topology. Obviously, $\Topdegree$ is finer than $\TopKrull$ and weaker than the discrete topology, let us see that these topologies can be all distinct.

Since the uniform structure is discrete, comparing these topologies amounts to answering questions of the form \enquote{if a $K$-automorphism of $L$ is trivial on $B$, is it also trivial on $C$?}, where $B$ and $C$ run among bounded sets of one kind or another. Thanks to the fundamental theorem of Galois theory \bbkiav{10}{7}{th.~4}, this question is easy to answer if $B$ and $C$ are Galois subextensions: the answer is then negative if and only if the subfield generated by $B$ and $C$, which is Galois by \bbkiav{10}{1}{prop.~1}, is distinct from $C$---that is, if $B$ is not contained in $C$. Let then $L_n$ be the subfield of $L$ generated by all elements whose degree is bounded by $n$. Observe that $L_n$ is a Galois subextension of $L$, since conjugated elements have the same degree. Therefore, comparing the above topologies can be rephrased as such:
\begin{itemize}
\item $\Topdegree$ is non-discrete if and only if $L \neq L_n$ for any natural number $n$;
\item $\Topdegree \neq \TopKrull$ if and only if there exists some $n$ such that $L_n$ is an infinite extension of $K$.
\end{itemize}
The first condition, that morally says that the wealth of algebraic numbers in $L$ cannot be exhausted by numbers of bounded degrees, can be hard to check on a specific infinite Galois extension (because, typically, $L_n$ contains elements of arbitrarily high degree, see for instance \cite[Th.~1.2]{GG_2014}). However, observe that the Galois group $\Gal(\extens{L_n}{K})$ has \emph{finite} exponent (which divides $(n!)!$), since an element of degree $\leq n$ has a decomposition field of degree $\leq n!$. In particular, $\Topdegree$ is non-discrete as soon as $\Gal(\extens{L}{K})$ has infinite exponent. This holds for instance if $K = \Rat$ and $L$ is an algebraic closure of $K$.

Endowed with $\TopKrull$ or $\Topdegree$, the group $\Gal(\extens{L}{K})$ is metrisable, Ra\u{\i}kov-complete, and non-Archimedean. Moreover, it is SIN by Proposition~\ref{prop:sin}. This is obvious for $\Topdegree$ since a $K$-automorphism preserves the degree of an algebraic element; for $\TopKrull$, recall that each element of $\Gal(\extens{L}{K})$ preserves Galois subextensions of $\extens{L}{K}$ and that any finite subextension is included in a finite Galois subextension \bbkiav{10}{1}{prop.~2}.

\begin{remark}
Since $L_n$ is generated (as a field) by bounded sets, it is also possible to define the topologies $\Topdegree$ as the projective limit of the Galois groups $\Gal(\extens{L_n}{K})$, each one endowed with the discrete topology (compare with the analoguous statement for $\TopKrull$ \bbkiav{10}{3}{prop.~5}).
\end{remark}

\begin{remark}
	Since $K$-automorphisms preserve the degree, we can easily define other bornologies by any arithmetic condition on the degree, such as
	\begin{quote}
		$B$ is \enquote{bounded} if there exists a finite set $S$ of prime numbers such that the degree of any element in $B$ is an $S$-integer\footnote{Recall that an integer $n$ is an \emph{$S$-integer} if the only primes that divide $n$ belong to $S$. (The number $1$ is thus in particular an $S$-integer for any $S$.)}.
	\end{quote}
	The topologies thus defined is (usually strictly) finer than $\Topdegree$, and enjoys the same topological properties. Whether it is non-discrete is however a trickier question.
\end{remark}

\section{Questions and future directions}

\subsection{Structure of \texorpdfstring{$\Aut(G)$}{Aut(G)}}
We have only lightly touched upon the structure of $\Aut(G)$ and many natural questions arise regarding its topological or geometric properties.

First, it would be interesting to have more informations on the bounded sets of $\Aut(G)$ and in particular to find criteria ensuring that $\Aut(G)$ is locally bounded for the lower or upper topology. When $G$ is locally compact, some criteria ensuring local \emph{compactness} of $\Aut(G)$ are known (see e.g.~\cite[I.6]{CM_2011}) but they rely ultimately on Arzela--Ascoli theorem, hence seem unlikely to be generalisable as such to boundedness.

Another particular class of interest is given by Roelcke-precompact groups. The structure of the automorphism group of \emph{compact} groups is relatively well understood (see e.g.~\cite[9.5]{HM_2013}); it would be interesting to know how much of it still holds for Roelcke-precompact groups, which is a natural generalisation.

\subsection{On inner automorphisms}

Let $G$ be a locally bounded group and $\conj\colon G \rightarrow \Aut(G)$ its conjugation morphism. Even in the locally compact case, the morphism $\conj$ may fail to have a closed image. This raises two kinds of questions:
\begin{enumerate}
\item If $\conj(G)$ is closed, what does that tell us about $G$?
\item In general, what does $\overline{\conj(G)}$ look like?
\end{enumerate}
 Caprace and Monod studied these two questions for locally compact groups in \cite[app.~I and II]{CM_2011} (to which we refer for more examples, results, and discussions) and unraveled their links to the structure theory of locally compact groups. It would be of the uttermost interest to know how much of their results carries to the more general locally bounded case, in particular for Polish groups.

As for the second question, let us just observe that if $G$ is centrefree or topologically simple, then the same is true for $\overline{\conj(G)}$ (the proof of \cite[I.5]{CM_2011} holds without any change) and that if $G$ is non-Archimedean, then $\Aut(G)$ (and thus also $\overline{\conj(G)}$), is also so by Remark~\ref{rem:nonarchgroup}.

As for the first question, we just record here a generalisation of \cite[I.4]{CM_2011}. Following \cite{CM_2011}, we will say that a locally bounded group is \emph{Ad-closed} if $\conj(G)$ is closed in $\Aut(G)$ for the lower topology. When $H$ is a subgroup of $G$, we denote by $\Zentrum_G (H)$ its centraliser in $G$ (that is, the subgroup of all elements in $G$ commuting with each element of $H$).

\begin{proposition}
 Let $G$ be a topological group and $H < G$ be a closed subgroup. Assume that $H$ is locally bounded and Ad-closed. Assume moreover either
\begin{enumerate}
 \item that $H$ is Ra\u{\i}kov-complete;
 \item or that the restriction of the lower uniform structure of $G$ to $H$ is the lower uniform structure of $H$.
\end{enumerate}
 Then $H \Zentrum_G (H)$ is closed in $G$.
\end{proposition}

\begin{proof}
 The proof follows closely \cite{CM_2011}. Let $G_0$ be the closure of $H \Zentrum_G (H)$. Since $H$ is closed, its normaliser in $G$ is also closed, and therefore contains $G_0$. In particular, $H$ is normal in $G_0$, which allows to define a morphism $\alpha\colon G_0 \rightarrow \Aut(H)$ by conjugation. We claim that either of the two hypotheses ensures that this morphism is continuous.

\begin{enumerate}
\item Indeed, observe that the \emph{restriction} $\beta$ of $\alpha$ to $H \Zentrum_G (H)$ factors through $H / \Zentrum (H)$, hence is continuous by Proposition~\ref{prop:conjcontlower}. Since it is a group morphism, it is thus upper uniformly continuous. Thanks to Theorem~\ref{th:raikovlower}, the group $\Aut(H)$ is complete for the upper uniform structure if $H$ is Ra\u{\i}kov-complete, hence we can extend continuously $\beta$ to the completion of $H \Zentrum_G (H)$ for the upper uniform structure, which includes $G_0$ (because the upper uniform structure is compatible with taking subgroups \cite[3.24]{RD_1981}).
\item On the other hand, if the lower uniform structure of $G$ induces that of $H$, then the same is true for the lower uniform structure of $G_0$ by transitivity of the induced uniform structures and by the fact that the uniform structure induced on a subgroup by the lower one is always coarser than the lower uniform structure of the subgroup. Therefore, the morphism $\alpha$ is also continuous in that case thanks to Remark~\ref{rem:conjcontroelckecomp}.
\end{enumerate}

We can now conclude the argument as in \cite{CM_2011}. Since $\Zentrum_G (H)$ acts trivially on $H$ by conjugation, $\alpha(H)$ is dense in $\alpha(G_0)$. But if $H$ is Ad-closed, $\alpha(H)$ is also closed in $\Aut(H)$. Therefore, $H \Zentrum_G (H)$, which is $\alpha\inv (\alpha(H))$, contains $G_0$, hence is closed.
\end{proof}

\bibliographystyle{../../BIB/amsalpha}
\bibliography{../../biblio/biblio}

\def\cprime{$'$}
\providecommand{\bysame}{\leavevmode\hbox to3em{\hrulefill}\thinspace}
\providecommand{\MR}{\relax\ifhmode\unskip\space\fi MR }
\providecommand{\MRhref}[2]{%
  \href{http://www.ams.org/mathscinet-getitem?mr=#1}{#2}
}
\providecommand{\href}[2]{#2}
\begin{thebibliography}{Bou81b}

\bibitem[Are46]{Arens_1946}
Richard Arens, \emph{Topologies for homeomorphism groups}, Amer. J. Math.
  \textbf{68} (1946), 593--610. \MR{19916}

\bibitem[BC46]{BC_1946}
Jean Braconnier and Jean Colmez, \emph{Sur les groupes d'hom\'{e}omorphismes
  d'un espace compl\`etement r\'{e}gulier}, C. R. Acad. Sci. Paris \textbf{223}
  (1946), 230--232. \MR{16662}

\bibitem[BDK06]{BDK_2006}
Sergey Bezuglyi, Anthony~H. Dooley, and Jan Kwiatkowski, \emph{Topologies on
  the group of {B}orel automorphisms of a standard {B}orel space}, Topol.
  Methods Nonlinear Anal. \textbf{27} (2006), no.~2, 333--385. \MR{2237460}

\bibitem[BK96]{BK_1996}
Howard Becker and Alexander~S. Kechris, \emph{The descriptive set theory of
  {P}olish group actions}, London Math. Soc. Lecture Note Ser., vol. 232,
  Cambridge University Press, Cambridge, 1996. \MR{1425877}

\bibitem[Bou71]{Bourbaki_TG_I}
N.~Bourbaki, \emph{{T}opologie g\'en\'erale. {C}hapitres \textup{\textsc{i}}
  \`a \textup{\textsc{iv}}}, Hermann, Paris, 1971. \MR{0358652}

\bibitem[Bou74]{Bourbaki_TG_V}
\bysame, \emph{Topologie g\'en\'erale. {C}hapitres \textup{\textsc{v}} \`a
  \textup{\textsc{x}}}, new ed., Hermann, Paris, 1974.

\bibitem[Bou81a]{Bourbaki_EVT}
\bysame, \emph{Espaces vectoriels topologiques. {C}hapitres \textup{\textsc{i}}
  \`a \textup{\textsc{v}}}, new ed., Masson, Paris, 1981. \MR{633754
  (83k:46003)}

\bibitem[Bou81b]{Bourbaki_A_IV}
\bysame, \emph{Algèbre. {C}hapitres \textup{\textsc{iv}} \`a  \textup{\textsc{vii}}}, Masson, Paris,
  1981. \MR{643362}

\bibitem[CM11]{CM_2011}
Pierre-Emmanuel Caprace and Nicolas Monod, \emph{Decomposing locally compact
  groups into simple pieces}, Math. Proc. Cambridge Philos. Soc. \textbf{150}
  (2011), no.~1, 97--128. \MR{2739075}

\bibitem[CH16]{CH_2016}
Yves~Cornulier and Pierre~de~la Harpe, \emph{{Metric Geometry of Locally Compact Groups}}, Tracts in Mathematics, vol.~25, European
  Mathematical Society, 2016. 

\bibitem[Die39]{Dieudonne_1939_complet}
J.~Dieudonn\'{e}, \emph{Sur les espaces uniformes complets}, Ann. \'{E}cole
  Norm. \textbf{56} (1939), 277--291. \MR{0001343}

\bibitem[Eng89]{Engelking_GT}
Ryszard Engelking, \emph{General topology}, second ed., Sigma Series in Pure
  Mathematics, vol.~6, Heldermann Verlag, Berlin, 1989, Translated from the
  Polish by the author. \MR{1039321}

\bibitem[GG14]{GG_2014}
Itamar Gal and Robert Grizzard, \emph{On the compositum of all degree $d$ extensions of a number field}, J. Th. Nombres Bordeaux \textbf{26} (2014), 655--672.

\bibitem[Gh19]{Gheysens_scatt}
Maxime~Gheysens, \emph{The homeomorphism group of the first uncountable ordinal}, available at
  \url{https://arxiv.org/abs/1911.09088}, 2019.

\bibitem[HM13]{HM_2013}
Karl~H. Hofmann and Sidney~A. Morris, \emph{The structure of compact groups},
  De Gruyter Studies in Mathematics, vol.~25, De Gruyter, Berlin, 2013, A
  primer for the student---a handbook for the expert, Third edition, revised
  and augmented. \MR{3114697}

\bibitem[Hod93]{Hodges_1993}
Wilfrid Hodges, \emph{Model theory}, Encyclopedia of Mathematics and its
  Applications, vol.~42, Cambridge University Press, Cambridge, 1993.
  \MR{1221741}

\bibitem[HR79]{HR_1979}
Edwin Hewitt and Kenneth~A. Ross, \emph{Abstract harmonic analysis. {Vol}.
  {I}}, second ed., Grundlehren der Mathematischen Wissenschaften, vol. 115,
  Springer-Verlag, Berlin-New York, 1979. \MR{551496}

\bibitem[Hye38]{Hyers_1938}
D.~H. Hyers, \emph{A note on linear topological spaces}, Bull. Amer. Math. Soc.
  \textbf{44} (1938), no.~2, 76--80. \MR{1563685}

\bibitem[Mai63]{Maissen_1963}
B.~Maissen, \emph{\"{U}ber {T}opologien im {E}ndomorphismenraum eines
  topologischen {V}ektorraumes}, Math. Ann. \textbf{151} (1963), 283--285.
  \MR{0154090}

\bibitem[NR12]{NR_2012}
Andrew Nicas and David Rosenthal, \emph{Coarse structures on groups}, Topology
  Appl. \textbf{159} (2012), no.~14, 3215--3228. \MR{2948279}

\bibitem[Pac13]{Pachl_2013}
Jan Pachl, \emph{Uniform spaces and measures}, Fields Institute Monographs,
  vol.~30, Springer, New York; Fields Institute for Research in Mathematical
  Sciences, Toronto, ON, 2013. \MR{2985566}

\bibitem[RD81]{RD_1981}
Walter Roelcke and Susanne Dierolf, \emph{Uniform structures on topological
  groups and their quotients}, Advanced Book Program, McGraw-Hill, New York,
  1981. \MR{644485}

\bibitem[Roe03]{Roe_coarse}
John Roe, \emph{Lectures on coarse geometry}, University Lecture Series,
  vol.~31, American Mathematical Society, Providence, RI, 2003. \MR{2007488}

\bibitem[Ros]{Rosendal_coarse}
Christian Rosendal, \emph{Coarse geometry of topological groups}, available at \url{http://homepages.math.uic.edu/~rosendal/}.

\bibitem[Tsa12]{Tsankov_2012}
Todor Tsankov, \emph{Unitary representations of oligomorphic groups}, Geom.
  Funct. Anal. \textbf{22} (2012), no.~2, 528--555. \MR{2929072}

\bibitem[Wei37]{Weil_uniforme}
André Weil, \emph{{Sur les espaces \`a structure uniforme et sur la topologie
  g\'en\'erale.}}, {Actual. sci. industr. 551, 39 p. (Publications de
  l'institut math\'ematique de l'universit\'e de Strasbourg) (1937).}, 1937.

\bibitem[Wei79]{Weil_collected_1}
Andr\'{e} Weil, \emph{Scientific works. {C}ollected papers. {V}ol. {I}
  (1926--1951)}, Springer-Verlag, New York-Heidelberg, 1979. \MR{537937}

\bibitem[Zie16]{Zielinski_2016}
Joseph Zielinski, \emph{An automorphism group of an {$\omega$}-stable structure
  that is not locally ({OB})}, MLQ Math. Log. Q. \textbf{62} (2016), no.~6,
  547--551. \MR{3601094}

\bibitem[Zie18]{Zielinski_2019}
\bysame, \emph{Locally {R}oelcke-precompact {P}olish groups}, available at
  \url{https://arxiv.org/abs/1806.03752}, 2018.

\end{thebibliography}

\end{document}